\documentclass[11pt,oneside,reqno]{amsart}

\usepackage{amssymb}
\usepackage{algpseudocode}
\usepackage{algorithm}
\usepackage{verbatim}
\usepackage{cite}
\RequirePackage{dsfont} 
 \scrollmode

\usepackage{graphicx}
\usepackage{epstopdf}
\usepackage[usenames]{color}
\definecolor{darkred}{RGB}{139,0,0}
\definecolor{darkgreen}{RGB}{0,100,0}
\definecolor{darkmagenta}{RGB}{139,0,139}
\long\def\symbolfootnote[#1]#2{\begingroup%
\def\thefootnote{\fnsymbol{footnote}}\footnote[#1]{#2}\endgroup}

\newcommand{\DEBUG}{}

\ifdefined\DEBUG%

  \def\rem#1{{\marginpar{\raggedright\scriptsize #1}}}
  \newcommand{\pmr}[1]{\rem{\color{blue}{$\bullet$ #1}}}
  \newcommand{\ppr}[1]{\rem{\color{red}{$\bullet$ #1}}}
 \else%

  \newcommand{\ppr}[1]{}
  \newcommand{\pmr}[1]{}
 \fi

\newcommand{\one}{\cdot\mathds{1}}

\newcommand{\xndfc}{{\tilde X_n^{df-M}}}

\newcommand{\xndfnc}{{\tilde {\bar X}_n^{df-RM}}}
\newcommand{\onei}{{\one_{[t_i, t_{i+1})}}}

\def\rho{\varrho}

\def\rd{\,{\mathrm d}}

\theoremstyle{plain}
\newtheorem{theorem}{Theorem}
\newtheorem{lemma}{Lemma}

\newtheorem{proposition}{Proposition}

\theoremstyle{definition}
\newtheorem{remark}{Remark}

\begin{document}

\title
[Randomized Milstein algorithm under inexact information]
{Randomized derivative-free Milstein algorithm for efficient approximation of solutions of SDEs under noisy information}

%\author[]{}
%\address{}
%\email{}

%\textit{AGH University of Science and Technology,\\
%Faculty of Applied Mathematics,\\
%Al. Mickiewicza 30, 30-059 Krakow, Poland,\\
%E-mail:} \texttt{pprzybyl@agh.edu.pl, morkiszp@student.agh.edu.pl}

\author[P. Morkisz]{Pawe{\l} M. Morkisz}
\address{AGH University of Science and Technology,
Faculty of Applied Mathematics,
Al. A.~Mickiewicza 30, 30-059 Krak\'ow, Poland\newline
and\newline
NVIDIA Poland
}
\email{morkiszp@agh.edu.pl}

\author[P. Przyby{\l}owicz]{Pawe{\l} Przyby{\l}owicz}
\address{AGH University of Science and Technology,
Faculty of Applied Mathematics,
Al. A.~Mickiewicza 30, 30-059 Krak\'ow, Poland}
\email{pprzybyl@agh.edu.pl, corresponding author}

%Pawe{\l} Przyby{\l}owicz \symbolfootnote[7]{}, \ Pawe\l \ Morkisz \vspace{.5cm}

\begin{abstract}
We deal with pointwise approximation of solutions  of scalar stochastic differential equations in the presence of informational noise about underlying drift and diffusion coefficients. We define a randomized derivative-free version of Milstein algorithm $\mathcal{\bar A}^{df-RM}_n$ and investigate its error. We also study lower bounds on the error of an arbitrary algorithm. It turns out that in some case the scheme $\mathcal{\bar A}^{df-RM}_n$ is the optimal one. Finally, in order to test  the algorithm $\mathcal{\bar A}^{df-RM}_n$ in practice, we report performed numerical experiments.

\end{abstract}
\maketitle

\tableofcontents
%%%%%%%%%%%%%%%%%%%%%%%%%%%%%%%%%%%%%%%%%%%%%%%%%%%%%%%%%%%%%%%%%%%%
%%%%%%%%%%%%%%%%%%%%%%%%%%%%%%
%%%%%%%%%%%%%%%%%%%%%%%%%%%%%%
\section{Introduction}\label{sec:intr}
In this paper we deal with pointwise approximation of solutions of the following scalar stochastic differential equations (SDEs) 
\begin{equation}
	\label{PROBLEM}
		\left\{ \begin{array}{ll}
			\rd X(t)=a(t,X(t))\rd t+ b (t,X(t))\rd W(t), &t\in [0,T], \\
			X(0)=\eta, 
		\end{array}\right.
\end{equation}
where $T>0$, $\eta$ is an initial-value, and $W=\{W(t)\}_{t\geq 0}$ is a standard one-dimensional Wiener process on some probability space $(\Omega,\Sigma,\mathbb{P})$. We will assume that only noisy evaluations of $a$ and $b$ are allowed. The aim is to find  an efficient approximation of $X(T)$ with the (asymptotic) error as small as possible.

The problem of approximation of solutions of SDEs under exact information about coefficients is well studied in the literature, see, for example, the standard reference \cite{KP}. Much less is known when values of drift and diffusion coefficients are corrupted by some noise. Therefore, in this paper we assume that evaluations of the underlying coefficients are permissible only at certain precision levels. Such a disturbance may be caused by, for example, measurement errors, rounding errors, and lowering a precision when performing computations on GPUs, see Remark \ref{GPU_model} and \cite{milvic}, \cite{Pla96} for further discussions and examples.

In literature there are many results on numerical problems under noisy information, such as integrating or approximation of  regular functions (\cite{JoC4, Pla96, KaPl90}), $L_p$ approximation of piecewise regular functions (\cite{MoPl16}), solutions of IVPs (\cite{KaPr16}) or PDEs (\cite{Wer96, Wer97}). For stochastic case we refer to \cite{PMPP17} and \cite{KaMoPr19} where the authors studied, respectively, approximation of SDEs under noisy information by randomized Euler scheme and approximate stochastic It\^o integration in the case when also the values of the Wiener process $W$ were inexact.

In this paper we extend the results  obtained in \cite{PMPP17}. Namely, we study approximation of solutions of SDEs by a randomized version of Milstein scheme under noisy information. For exact information such a version of the Milstein scheme was investigated in \cite{KRWU}. Here, however, we use its derivative free version in order to cover also the case of inexact information. Hence, our proof technique differs from that used in \cite{KRWU}.

We use a suitable computation setting that allows us to model the situation  when the values of $a$'s and $b$'s are perturbed by some deterministic noise, see \cite{PMPP17}.  Namely, let  $\delta_1, \delta_2 \in [0,1]$ be the precision levels corresponding to drift and the diffusion coefficients, respectively. (The case of $\delta_1 = \delta_2 = 0$ corresponds to the exact information.) Available standard information about each coefficient consists of  noisy evaluations of the coefficients at a finite number of points $(t_i, y_i) \in [0,T] \times \mathbb{R}$. This means that, for example, for the diffusion coefficient $b$ and for a given point $(t_i,y_i)\in[0,T]\times\mathbb{R}$  evaluation returns  $\tilde b(t_i,y_i)$ with the property that $| b(t_i,y_i) - \tilde b(t_i,y_i) | \leq \delta_2(1+|y_i|)$.  Moreover, as in \cite{PMPP17} for $a = a(t, y)$ we allow randomized choices of sample points with respect to the time variable $t$. For the Wiener process $W$ we assume that the information is exact, i.e.: it  is given by the values of $W$ at a finite number of points $s_k\in [0,T]$. (See, however, Remark \ref{inexact_Wiener}.) The error of an algorithm, using the information above, is measured in the $q$-th mean ($q\geq 1$) maximized over the class of input data $(a,b,\eta)$ and over all permissible information about $(a,b,\eta)$ with the given precisions $\delta_1, \delta_2\geq 0$.

Theorem \ref{main_theorem}, which is the main result of the paper,  states that the $n$th minimal error (under suitably regular informational noise)   is asymptotically equal to $\Theta(n^{-\min\{\frac{1}{2}+\gamma_1,\gamma_2\}} + \delta_1 + \delta_2)$ where the factors in $\Theta$  do not depend on $\delta_1$, $\delta_2$. (Here, $\gamma_1,\gamma_2\in (0,1]$ are the H\"older exponents, with respect to time variable, of drift and diffusion coefficients, respectively.) A randomized derivative-free version $\mathcal{\bar A}_n^{df-RM}$ of the classical Milstein  algorithm is defined which uses  noisy evaluations of the drift and diffusion coefficients, and attains the desired rate of convergence. When the disturbances for $a$ and $b$ are more rough, then  error term for the scheme $\mathcal{\bar A}_n^{df-RM}$ also depends on $\delta_2n^{1/2}$, see Theorem \ref{RMDF_NOIZZ_ERR} (ii). This implies that in order to obtain any convergence rate it is necessary to tend with both precision levels to zero suitably fast with respect to $n$. 

The paper is organized as follows. Section \ref{sec:prel} consists of the problem formulation, basic notions and definitions.
Randomized derivative-free  Milstein algorithm $\mathcal{\bar A}_n^{df-RM}$ under perturbed information together with upper bounds on its error are presented in Section \ref{sec:alg}. In Section \ref{sec:low} we show a lower bound on worst case error for an arbitrary algorithm  (Lemma \ref{lower_b_lem}). This leads to the conclusion that the randomized Milstein algorithm $\mathcal{\bar A}_n^{df-RM}$ is optimal (Theorem \ref{main_theorem}). Section \ref{sec:num} reports  numerical experiments performed for the algorithm $\mathcal{\bar A}_n^{df-RM}$. Finally, the Appendix contains an auxiliary facts used in the paper.
%%%%%%%%%%%%%
%%%%%%%%%%%%%
\section{Preliminaries} \label{sec:prel}
Let $T>0$. We denote by $\mathbb{N}=\{1,2,\ldots\}$. Let $W=\{W(t)\}_{t\geq 0}$ be a standard one-dimensional Wiener process on a complete probability space $(\Omega,\Sigma,\mathbb{P})$. We denote by $\{\Sigma_t\}_{t\geq 0}$ a filtration, satisfying the usual conditions, such that $W$ is a Wiener process on $(\Omega,\Sigma,\mathbb{P})$ with respect to $\{\Sigma_t\}_{t\geq 0}$. Let $\displaystyle{\Sigma_{\infty}=\sigma\Bigl(\bigcup_{t\geq 0}\Sigma_t\Bigr)}$. For a random variable $X:\Omega\to\mathbb{R}$ we write $\|X\|_q=(\mathbb{E}|X|^q)^{1/q}$, where $q\in [2,+\infty)$. For any $f\in C^{0,1}([0,T]\times\mathbb{R})$ by $L_1$ we mean the following differential operator  
\begin{displaymath}
 L_1 f(t,y) = f(t,y)\cdot \frac{\partial f}{\partial y}(t,y).
\end{displaymath}
We will also use its derivative-free version. Namely, for $f\in C([0,T]\times\mathbb{R})$ and $h>0$  the difference operator $\mathcal{L}_{1,h}$ is given as follows
\begin{displaymath}
	\mathcal{L}_{1,h}f(t,y)=f(t,y)\cdot\Delta_h f(t,y),
\end{displaymath}
where
\begin{displaymath}
\Delta_h f(t,y)=\frac{f(t,y+h)-f(t,y)}{h}.
\end{displaymath}
(Basic properties of $L_1$ and $\mathcal{L}_{1,h}$, used in the paper, are gathered in Appendix.)

Let $K>0$ and $\gamma\in (0,1]$. We say that $f:[0,T] \times \mathbb{R} \to \mathbb{R}$ belongs to the function class $F_K^{\gamma}$ iff for all $t,s\in[0,T]$ and all $y,z\in\mathbb{R}$ it satisfies the following assumptions:
\begin{itemize}
	\item [(i)] $f\in C^{0,2}\left([0,T]\times\mathbb{R}\right)$,
	\item [(ii)] $|f(0,0)| \leq K$,
	\item [(iii)] $|f(t,y) - f(t,z)| \leq K |y-z|$,
	\item [(iv)] $f(t,y) - f(s,y)| \leq K(1+|y|)|t-s|^{\gamma}$,
	\item [(v)] $\left| \frac{\partial f}{\partial y}(t,y) - \frac{\partial f}{\partial y}(t,z) \right| \leq K |y-z|$.
\end{itemize}
In this paper we will be considering drift coefficients $a$ from the following class
\begin{displaymath}
	\mathcal{A}^{\gamma_1}_K=\Biggl\{f \in F^{\gamma_1}_K \,\Bigl|\, \left| \frac{\partial f}{\partial y}(t,y) - \frac{\partial f}{\partial y}(s,y) \right| \leq K (1+|y|)|t-s|^{\gamma_1} \ \hbox{for all} \ t,s\in [0,T], y\in\mathbb{R}\Biggr\},
\end{displaymath}
while we will be assuming that diffusion coefficients are from 
\begin{displaymath}
 \mathcal{B}^{\gamma_2}_K = \Bigl\{ f \in F^{\gamma_2}_K \,\Bigl|\, |L_1 f(t,y) - L_1 f(t,z)|\leq K|y-z| \ \hbox{for all} \ t\in [0,T], y,z\in\mathbb{R}\Bigr\}.
 \end{displaymath}
Moreover, let
\begin{displaymath}
	\mathcal{J}^q_{K}=\{\eta:\Omega\to\mathbb{R} \ | \ \eta \ \hbox{is} \ \Sigma_0-\hbox{measurable}, \mathbb{E}|\eta|^{2q}\leq K\}.
\end{displaymath}
For $\gamma_1,\gamma_2\in (0,1]$, $q\in [2,+\infty)$, $K\in (0,+\infty)$ we consider the following class of admissible input data
\begin{displaymath}
	\mathcal{F}(\gamma_1,\gamma_2,q,K)=\mathcal{A}^{\gamma_1}_K\times\mathcal{B}^{\gamma_2}_K\times \mathcal{J}^q_{K}.
\end{displaymath}
For all $(a, b,\eta) \in \mathcal{F}(\gamma_1,\gamma_2,q,K)$ the equation \eqref{PROBLEM} has a unique strong solution $\{X(t)\}_{t\in [0,T]}$, that is adapted to $\{\Sigma_t\}_{t\in [0,T]}$, see, for example, \cite{KARSHR}. The numbers $T, K, q, \gamma_1, \gamma_2$ will be called parameters of the class $\mathcal{F}(\gamma_1,\gamma_2,q,K)$. Except for $T$ the parameters are, in general, not known and the algorithms presented later on will not use them as input parameters.

Under some minor modifications, we recall from \cite{PMPP17} a model of computation under inexact information about $a$'s and $b$'s. To do that we need to introduce the following auxiliary classes:
\begin{displaymath}
 \mathcal{K}^1 = \{ p:[0,T]\times\mathbb{R}\to\mathbb{R} \,|\, p(\cdot,\cdot) - \mbox{ Borel measurable }, |p(t,y)| \leq 1+|y|, t\in[0,T], y\in\mathbb{R}\},
\end{displaymath}
\begin{displaymath}
 \mathcal{K}^1_{\rm Lip} = \{ p\in\mathcal{K}^1 \,|\, |p(t,y)-p(t,z)| \leq |y-z|, t\in[0,T], y,z\in\mathbb{R}\},
\end{displaymath}
and
\begin{displaymath}
\mathcal{K}^2 = \{ p \in \mathcal{K}^1 \,|\, |p(t,y)| \leq 1, t\in[0,T], y\in\mathbb{R}\},
\end{displaymath}
see also \cite{KaMoPr19}. The classes $ \mathcal{K}^1_{\rm Lip}$, $\mathcal{K}^2$  are nonempty and contain constant functions. (This is an important fact  from a point of view of lower error bounds, see \cite{PMPP17}.) Let $\delta_1$, $\delta_{2}\in [0,1]$. We refer to $\delta_1$, $\delta_2$ as to precision parameters. For $a\in\mathcal{A}^{\gamma_1}_K$ we define the following class of corrupted drift coefficients
\begin{displaymath}
 V_a(\delta_1) = \{ \tilde a \,|\, \exists_{p_a \in \mathcal{K}^1} \,:\, \tilde a = a + \delta_1\cdot p_a \},
\end{displaymath}
while for  $b\in\mathcal{B}^{\gamma_2}_K$  we consider the following two classes of corrupted diffusion coefficients
\begin{displaymath}
 V^1_b(\delta_2) = \{ \tilde b \,|\, \exists_{p_b \in \mathcal{K}^1_{\rm Lip}} \,:\, \tilde b = b + \delta_2\cdot p_b \},
\end{displaymath}
and
\begin{displaymath}
 V^2_b(\delta_2) = \{ \tilde b \,|\, \exists_{p_b \in \mathcal{K}^2} \,:\, \tilde b = b + \delta_2\cdot p_b \}.
\end{displaymath}
Note that we impose more smoothness for corrupting functions $p_b$'s than for $p_a$'s. This is due to some technicalities, see Remark \ref{smooth_pb}.
We have that $V_a(\delta_1)\subset V_a(\delta_1^{'})$ for $0\leq \delta_1\leq\delta_1^{'}\leq 1$, and $V^i_b(\delta_2)\subset V^i_b(\delta_2^{'})$ for $0\leq \delta_2\leq\delta_2^{'}\leq 1$, $i=1,2$. We assume that the algorithm is based on discrete noisy information about $(a,b)$ and exact information about $W$, and $\eta$. Hence, a vector of noisy information has the following form
\begin{eqnarray*}
	\mathcal{N}(\tilde z(\delta_1,\delta_2))&=&\Bigl[\tilde a (\xi_0,y_0),\tilde a(\xi_1,y_1),\ldots,\tilde a(\xi_{i_1-1},y_{i_1-1}),\notag\\
		 && \ \  \tilde b(t_0,z_0), \tilde b(t_1,z_1),\ldots, \tilde b(t_{i_1-1},z_{i_1-1}), \notag\\
		 && \ \ \tilde b(t_0,u_0), \tilde b(t_1,u_1),\ldots, \tilde b(t_{i_1-1},u_{i_1-1}), \notag\\
		&& \ \ W(s_0),W(s_1),\ldots, W(s_{i_2-1}),\eta],
\end{eqnarray*}
where  $i_1,i_2\in\mathbb{N}$ and $(\xi_0,\xi_1,\ldots,\xi_{i_1-1})$ is a random vector on $(\Omega,\Sigma,\mathbb{P})$ which takes values in $[0,T]^{i_1}$. We assume that the $\sigma$-fields $\sigma(\xi_0,\xi_1,\ldots,\xi_{i_1-1})$ and $\Sigma_{\infty}$ are independent. Moreover,   $t_0,t_1,\ldots,t_{i_1-1}\in [0,T]$ and $s_0,s_1,\ldots,s_{i_2-1}\in [0,T]$ are given time points.
We  assume that $t_i\neq t_j$, $s_i\neq s_j$ for all $i\neq j$. The evaluation points $y_j$, $z_j$, $u_j$ for the spatial variables $y,z$ of $a(\cdot,y)$, $b(\cdot,z)$, and $b(\cdot,u)$ can be given in adaptive way with respect to $(a,b,\eta)$ and $W$. This means that  there exist Borel measurable mappings $\psi_0:\mathbb{R}^{i_2}\times\mathbb{R}\to\mathbb{R}^{3}$, $\psi_j:\mathbb{R}^{j}\times\mathbb{R}^{j}\times\mathbb{R}^{j}\times\mathbb{R}^{i_2}\times\mathbb{R}\to\mathbb{R}^{3}$, $j=1,2,\ldots,i_1-1$, such that the successive points $y_j,z_j$ are computed in the following way:
\begin{displaymath}
	(y_0,z_0,u_0)=\psi_0\Bigl(W(s_0),W(s_1),\ldots, W(s_{i_2-1}),\eta\Bigr), 
\end{displaymath}
where 
\begin{eqnarray*}
	(y_j,z_j,u_j)&=&\psi_j\Bigl(\tilde a(\xi_0,y_0), \tilde a(\xi_1,y_1),\ldots, \tilde a(\xi_{j-1},y_{j-1}),\notag\\
	&& \ \ \quad \tilde b(t_0,z_0), \tilde b(t_1,z_1),\ldots, \tilde b(t_{j-1},z_{j-1}),\notag\\
	&& \ \ \quad \tilde b(t_0,u_0), \tilde b(t_1,u_1),\ldots, \tilde b(t_{j-1},u_{j-1}),\notag\\
	&& \ \ \quad W(s_0), W(s_1),\ldots, W(s_{i_2-1}),\eta\Bigr), 
\end{eqnarray*}
for $j=1,2,\ldots, i_1-1$. The total number of (noisy) evaluations of $a$, $b$ and $W$ is equal to $l=3i_1+i_2$.
\\
Any algorithm $\mathcal{A}$ using $\mathcal{N}(\tilde a, \tilde b,\eta,W)$, that computes approximation to $X(T)$ is of the form
\begin{equation}
	\label{def_alg}
	\mathcal{A}(\tilde a, \tilde b,\eta,W,\delta_1,\delta_2)=\varphi(\mathcal{N}(\tilde a, \tilde b,\eta,W)),
\end{equation}
for some Borel measurable mapping $\varphi:\mathbb{R}^{3i_1+i_2+1}\to\mathbb{R}$. For a given $n\in\mathbb{N}$ we denote by $\Phi_n$ a class of all algorithms of the form \eqref{def_alg} for which the total number of evaluations $l$
is at most $n$. 

For a fixed $(a,b,\eta,)\in\mathcal{F}(\gamma_1,\gamma_2,q,K)$ the error of $\mathcal{A}\in\Phi_n$ is defined in the following way
\begin{displaymath}
e^{(q)} (\mathcal{A}, a,b,\eta, W, V^i, \delta_1, \delta_2) = \sup_{(\tilde a,\tilde b)\in V_a(\delta_1)\times V^i_b(\delta_2)} \| X(T) - \mathcal{A}(\tilde a, \tilde b, \eta, W, \delta_1,\delta_2)\|_{q}.
\end{displaymath}
for $i=1,2$, where $X(T)=X(a,b,\eta,W)(T)$. The worst-case error of the algorithm $\mathcal{A}$ is defined as
\begin{displaymath}
	e^{(q)} (\mathcal{A}, \mathcal{F}, W, V^i, \delta_1, \delta_2)=\sup\limits_{(a,b,\eta)\in\mathcal{F}}e^{(q)} (\mathcal{A}, a,b,\eta, W, V^i, \delta_1, \delta_2),
\end{displaymath}
where $\mathcal{F}$ is a certain subclass of $\mathcal{F}(\gamma_1,\gamma_2,q,K)$, see \cite{Pla96} and \cite{TWW88}. Hence, we are considering the worst error with respect to any $(\tilde a, \tilde b)$ that can be given to us for a fixed $(a,b)$. Finally, we define the $n$th minimal error as follows
\begin{displaymath}
	e_n^{(q)}(\mathcal{F}, W, V^i, \delta_1, \delta_2)=\inf\limits_{\mathcal{A}\in\Phi_n}e^{(q)} (\mathcal{A}, \mathcal{F}, W, V^i, \delta_1, \delta_2), \ i=1,2.
\end{displaymath}
Our aim is to find possibly sharp bounds on the $n$th minimal error $e_n^{(q)}(\mathcal{F}, W, V^i, \delta_1, \delta_2)$, i.e., lower and upper bounds which match up to constants. We are also interested in defining an algorithm for which the infimum in $e_n^{(q)}(\mathcal{F}, W, V^i, \delta_1, \delta_2)$ is asymptotically attained.

Unless otherwise stated, all constants appearing in this paper (including those in the "O", "$\Omega$", and "$\Theta$" notation) will only
depend on the parameters of the respective classes. Furthermore, the same symbol may be used for different constants.
%%%%%%%%%%%%%%
\begin{remark}
\label{GPU_model}  It is worth mentioning that the proposed computation and error setting includes the phenomenon of lowering precision of computations. Namely, we can model relative roundoff errors by considering disturbing functions $p_f$, $f\in\{a,b\}$, of the form
\begin{equation}
	p_f(t,y)=\alpha(t,y)\cdot f(t,y), \ (t,y)\in [0,T]\times\mathbb{R},
\end{equation}
for some function $\alpha$ that is Borel measurable and bounded on $[0,T]\times\mathbb{R}$. That is a frequent case for efficient computations using both CPUs and GPUs. An example could be the current state-of-the-art GPU - NVIDIA Tesla V100, which performance behaves as follows - 7 TeraFLOPS for {\it double precision}, 14 TeraFLOPS for {\it single precision}, and up to 112 TeraFLOPS for {\it half precision} of very specific type (repeatable operations of matrix multiplications and additions). We refer to \cite{KaMoPr19} where  Monte Carlo simulations were performed on GPUs.
\end{remark}
%%%%%%%%%%%%%%
\section{Randomized derivative-free Milstein algorithm for noisy information} \label{sec:alg}
Below we define the randomized derivative-free Milstein algorithm in presence of informational noise for $a$ and $b$. 
Let $n\in\mathbb{N}$ and let
\begin{equation}
	\label{discret}
	t_i = ih, \quad h=T/n, \quad i=0,1,\ldots,n,
\end{equation}
be the equidistant discretization on $[0,T]$. Moreover, we take
\begin{displaymath}
 \Delta W_i = W(t_{i+1}) - W(t_i), 
\end{displaymath}
\begin{displaymath}
	\mathcal{I}_{t_i, t}(W,W) = \int\limits_{t_i}^{t} \int\limits_{t_i}^{s} \mathrm{d}W(u) \mathrm{d}W(s) = \frac{1}{2} \Bigl((W(t)-W(t_i))^2 - (t-t_i)\Bigr),
\end{displaymath}
for $t\in [t_i,t_{i+1}]$, $i=0,1,\ldots,n-1$. Let $\{\xi_i\}_{i=0}^{n-1}$ be independent random variables on the probability space $(\Omega,\Sigma,\mathbb{P})$, such that the $\sigma$-fileds $\sigma(\xi_0,\xi_1,\ldots,\xi_{n-1})$ and $\Sigma_{\infty}$ are independent, with $\xi_i$ being uniformly distributed on $[t_i,t_{i+1}]$.
Then for any $(a,b,\eta) \in \mathcal{F}(\gamma_1,\gamma_2,q,K)$, $(\tilde a, \tilde b)\in V_a(\delta_1)\times (V^1_b(\delta_2)\cup V^2_b(\delta_2))$ we set
\begin{equation}
	\label{MIL_DF_NO}
		\left\{ \begin{array}{ll}
			\bar X^{df-RM}_n(0)  & =   \eta, \\
			\bar X^{df-RM}_n(t_{i+1})  & = \bar X^{df-RM}_n(t_i) + \tilde a(\xi_i, \bar X^{df-RM}_n(t_i)) \cdot\frac{T}{n} +  \tilde b(t_i, \bar X^{df-RM}_n(t_i))  \cdot \Delta W_i \\
				& +\mathcal{L}_{1,h}\tilde b(t_i, \bar X^{df-RM}_n(t_i))\cdot\mathcal{I}_{t_i, t_{i+1}}(W,W),
		\end{array}\right.
\end{equation}
for $i=0,1,\ldots,n-1$, where
\begin{displaymath}
	\mathcal{L}_{1,h}\tilde b(t_i, \bar X^{df-RM}_n(t_i))=\tilde b(t_i, \bar X^{df-RM}_n(t_i))\cdot\frac{\tilde b(t_i, \bar X^{df-RM}_n(t_i)+h)-\tilde b(t_i, \bar X^{df-RM}_n(t_i))}{h}.
\end{displaymath}
 The algorithm $\mathcal{\bar A}^{df-RM}_n$ is defined as
\begin{equation}
\label{RMDF_NOIZ_DEF}
	\mathcal{\bar A}^{df-RM}_n(\tilde a,\tilde b,\eta,W,\delta_1,\delta_2) := \bar X^{df-RM}_n(T). 
\end{equation} 
In the case of exact information (i.e., $\delta_1=\delta_2=0$) we write $X^{df-RM}_n$ and $\mathcal{A}^{df-RM}_n$ instead of $\bar X^{df-RM}_n$, and $\mathcal{\bar A}^{df-RM}_n$, respectively. The total number of evaluations of $a$, $b$, and $W$ used for computing $\mathcal{\bar A}^{df-RM}_n$ is $4n$. Therefore, $\mathcal{\bar A}^{df-RM}_n\in\Phi_{4n}$. Moreover, the combinatorial cost consists of $O(n)$ arithmetic operations.

In the following theorem we state upper bounds on the error of the randomized derivative-free Milstein scheme under noisy information about $a$ and $b$. 
\begin{theorem}
	\label{RMDF_NOIZZ_ERR}
	\begin{itemize}
		\item [(i)] There exists a positive constant $C$, depending only on the parameters of the class $\mathcal{F}(\gamma_1,\gamma_2,q,K)$, such that for all $n\in\mathbb{N}$, $\delta_1, \delta_2\in [0,1]$, $(a,b,\eta)\in \mathcal{F}(\gamma_1,\gamma_2,q,K)$, $(\tilde a, \tilde b)\in V_a(\delta_1)\times V^1_b(\delta_2)$ , we have
		\begin{displaymath}
		\|X(a,b,\eta,W)(T)-\mathcal{\bar A}_n^{df-RM}(\tilde a,\tilde b,\eta,W,\delta_1, \delta_2)\|_q \leq C\cdot\Bigl(n^{-\min\{\frac{1}{2}+\gamma_1,\gamma_2\}}+\delta_1+\delta_2\Bigr).
	\end{displaymath}
		\item [(ii)] There exist positive constants $C_1,C_2,C_3$, depending only on the parameters of the class $\mathcal{F}(\gamma_1,\gamma_2,q,K)$ and $q$, such that for all $n\in\mathbb{N}$, $\delta_1,\delta_2\in [0,1]$, $(a,b,\eta)\in \mathcal{F}(\gamma_1,\gamma_2,q,K)$, $(\tilde a, \tilde b)\in V_a(\delta_1)\times V^2_b(\delta_2)$ , we have
\begin{eqnarray*}
		&&\|X(a,b,\eta,W)(T)-\mathcal{\bar A}_n^{df-RM}(\tilde a,\tilde b,\eta,W,\delta_1, \delta_2)\|_q \leq C_1\cdot n^{-\min\{\frac{1}{2}+\gamma_1,\gamma_2\}}\\
&&+C_2 \cdot e^{C_3(\delta_2n^{1/2})^{q}} \cdot (1+\delta_2n^{1/2}) \cdot \Bigl(\delta_1+\delta_2 n^{1/2}+n^{-3/2}\Bigr).
	\end{eqnarray*}
	\end{itemize}
\end{theorem}

The aim of this section is to justify Theorem \ref{RMDF_NOIZZ_ERR}. Before we do that we need to prove several auxiliary results concerning, in particular, upper bounds on error of the following time-continuous version of the randomized derivative-free Milstein algorithm $\mathcal{\bar A}^{df-RM}_n$ in presence of noise. Namely, let us take
\begin{equation}
	\label{MIL_DFNC}
		\left\{ \begin{array}{ll}
			\xndfnc(0)  & =   \eta, \\
			\xndfnc (t)  & = \xndfnc(t_i) + \tilde a(\xi_i, \xndfnc(t_i)) \cdot (t-t_i) \\
				&\quad +  \tilde b(t_i, \xndfnc(t_i))  \cdot (W(t)-W(t_i)) \\
				&\quad +\mathcal{L}_{1,h}\tilde b(t_i, \xndfnc(t_i))\cdot\mathcal{I}_{t_i, t}(W,W),
		\end{array}\right.
\end{equation}
for $t\in [t_i,t_{i+1}]$ and $i=0,1,\ldots,n-1$. In the case of exact information we write $\tilde X_n^{df-RM}=\{\tilde X_n^{df-RM}(t)\}_{t\in [0,T]}$ instead of $\xndfnc =\{ \xndfnc (t) \}_{t\in[0,T]}$. It holds $\xndfnc(t_i)=\bar X^{df-RM}_n(t_i)$ for $i=0,1,\ldots,n$. Hence, it is sufficient to analyze the error of $\xndfnc$. We also extend the filtration $\{\Sigma_t\}_{t\geq 0}$ in the same way as in \cite{PMPP17}. Namely, let $\mathcal{G}^n=\sigma(\xi_0,\xi_1,\ldots,\xi_{n-1})$ and $\tilde\Sigma_t^n=\sigma\Bigl(\Sigma_t\cup\mathcal{G}^n\Bigr)$, $t\geq 0$. Since the $\sigma$-fields $\Sigma_{\infty}$ and $\mathcal{G}^n$ are independent, the process $W$ is still a one-dimensional Wiener process on $(\Omega,\Sigma,\mathbb{P})$ with respect to $\{\tilde\Sigma_t^n\}_{t\geq 0}$. In the sequel we will consider stochastic It\^o integrals with respect to $W$ of processes that are adapted to the filtration $\{\tilde\Sigma_t^n\}_{t\geq 0}$. In particular, the following technical lemma assures suitable measurability of the process $\xndfnc =\{ \xndfnc (t) \}_{t\in[0,T]}$ with respect to $\{\tilde\Sigma_t^n\}_{t\geq 0}$.
%%%%%%%%%%%%%%%%%%%% 
\begin{lemma}
	\label{PROG_RMDFNO_PROC} Let $n\in\mathbb{N}$, $\delta_1,\delta_2\in [0,1]$, $(a,b,\eta)\in\mathcal{F}(\gamma_1,\gamma_2,q,K)$ and $(\tilde a,\tilde b)\in V_a(\delta_1)\times (V^1_b(\delta_2)\cup V^2_b(\delta_2))$. Then the process  $\xndfnc =\{ \xndfnc (t) \}_{t\in[0,T]}$ is progressively measurable with respect to the filtration $\{  \tilde\Sigma_t^n\}_{t\geq 0}$.
\end{lemma}
The lemma above follows from induction and Proposition 1.13 in \cite{KARSHR}. Hence, we skip its proof.

In order to justify Theorem \ref{RMDF_NOIZZ_ERR} we proceed as follows. First, in Section \ref{sec:alg_RandMilstein} we investigate the error of the randomized version of the classical Milstein algorithm when  information about $a$ and $b$ is exact. Then, in Section \ref{sec:Perf_randdfMilstein} we show upper bounds for the derivative-free version of the randomized Milstein scheme also for exact information about $a$ and $b$. Finally, combining results obtained for these two methods we show the upper bounds on the error of $\bar X^{df-RM}$ in the presence of informational noise. 
%%%%%%%%%%%%
\begin{remark}
\label{smooth_pb}
It is natural to ask about a version of Theorem \ref{RMDF_NOIZZ_ERR}  when  corrupting functions $p_b$ are from $\mathcal{K}^1$, as it is for $p_a$'s. However, in this case we were unable to show any nontrivial upper bound for the algorithm $\mathcal{\bar A}_n^{df-RM}$. It turns out that for $p_b\in\mathcal{K}^1$ the function $(t,y)\to\mathcal{L}_{1,h}\tilde b(t,y)$ might be of super-linear growth with respect to $y$. Hence, we conjecture that some modification of the scheme $\mathcal{\bar A}_n^{df-RM}$ is needed in order to obtain analogous bounds as in Theorem \ref{RMDF_NOIZZ_ERR}. We postpone this problem to our future work.
\end{remark}
%%%%%%%%%%%%%%
\begin{remark}
\label{inexact_Wiener}
In \cite{KaMoPr19} the authors considers approximate stochastic It\^o integration in the case when the values of the Wiener process are corrupted by informational noise. Preliminary estimates suggest that direct application of techniques used in \cite{KaMoPr19} to approximation of SDEs, under inexact information about $W$, is not possible. Thereby, further investigation in that direction is  needed. 
\end{remark}
%%%%%%%%%%%%%%
\subsection{Performance of randomized  Milstein algorithm for exact information} \label{sec:alg_RandMilstein}
\noindent\newline
By randomizing evaluations of the drift coefficient $a$ in the classical Milstein scheme, we arrive at the following randomized Milstein algorithm. Take
\begin{equation}
	\label{RANMIL}
		\left\{ \begin{array}{ll}
			  X_n^{RM}(0)  & =   \eta, \\
			  X_n^{RM} (t_{i+1})  & = X_n^{RM}(t_i) + a(\xi_i, X_n^{RM}(t_i)) \cdot\frac{T}{n} + b(t_i, X_n^{RM}(t_i))  \cdot \Delta W_i \\
				& +L_1b(t_i, X_n^{RM}(t_i))\cdot \mathcal{I}_{t_i, t_{i+1}}(W,W),
		\end{array}\right.
\end{equation}
for $i=0,1,\ldots,n-1$. The algorithm $\mathcal{A}^{RM}_n$ is defined as
\begin{equation}
\label{RM_DEF}
	\mathcal{A}^{RM}_n(a,b,\eta,W,0,0) := X^{RM}_n(T). 
\end{equation} 
Note that $\mathcal{A}^{RM}_n\notin\Phi_n$, since it uses values  of the partial derivative of $b$. We refer to $\mathcal{A}^{RM}_n$ as to an auxiliary method that helps us to estimate the error of $\mathcal{\bar A}_n^{df-RM}$.
 
In order to investigate the error of the method $\mathcal{A}^{RM}_n$ we define the following time-continuous version of the scheme $X^{RM}_n$ as follows:
\begin{equation}
	\label{RANMIL_CONT}
		\left\{ \begin{array}{ll}
			\tilde X_n^{RM}(0)  =  &  \eta, \\
			\tilde X_n^{RM} (t)  = & \tilde X_n^{RM}(t_i) + a(\xi_i, \tilde X_n^{RM}(t_i)) \cdot (t-t_i) \\
&+ b(t_i, \tilde X_n^{RM}(t_i)) \cdot (W(t) - W(t_i)) \\ 
& +L_1 b(t_i, \tilde X_n^{RM}(t_i))\cdot \mathcal{I}_{t_i, t}(W,W),
		\end{array}\right.
\end{equation}
for $t\in[t_i, t_{i+1}]$ and for $i=0,1,\ldots,n-1$. We have that  $\tilde X_n^{RM} (t_i) = X_n^{RM} (t_i)$ for all $0\leq i \leq n$. Hence, it is sufficient to analyze the error of time-continuous version of the algorithm. Moreover, for the process $\{\tilde X_n^{RM} (t)\}_{t\in [0,T]}$ the following version of Lemma \ref{PROG_RMDFNO_PROC} holds.
\begin{lemma}
	\label{PROG_RM_PROC} Let $n\in\mathbb{N}$, $(a,b,\eta)\in\mathcal{F}(\gamma_1,\gamma_2,q,K)$. Then the process  $\tilde X^{RM}_n=\{ \tilde X_n^{RM}(t) \}_{t\in[0,T]}$ is progressively measurable with respect to the filtration $\{\tilde\Sigma_t^n\}_{t\geq 0}$.
\end{lemma}

We have the following result for the algorithm $\mathcal{A}^{RM}_n$.
\begin{proposition}
	\label{PROP_RM_ERR} 
There exists a positive constant $C$, depending only on the parameters of the class $\mathcal{F}(\gamma_1,\gamma_2,q,K)$, such that for all $n\in\mathbb{N}$ and all $(a,b,\eta)\in\mathcal{F}(\gamma_1,\gamma_2,q,K)$ we have
\begin{equation}
\label{est_XXRM}
	\sup\limits_{t\in [0,T]}\|X(t)-\tilde X^{RM}_n(t)\|_q\leq Cn^{-\min\{\frac{1}{2}+\gamma_1,\gamma_2\}},
\end{equation}
and, in particular, 
	\begin{displaymath}
		\|X(T)-\mathcal{A}_n^{RM}(a,b,\eta,W,0,0)\|_q \leq C n^{-\min\{\frac{1}{2}+\gamma_1,\gamma_2\}}.
	\end{displaymath}
\end{proposition}
{\bf Proof.} We show upper bound for $ \sup_{t\in[0,T]} \|X(t) - \tilde X_n^{RM} (t)\|_q$, from which the desired result follows.

The solution $X$ can be expressed in the following way
\begin{align*}
	X(t) & = \eta + A(t) + B(t), \\
	A(t) & = \int\limits_0^t \sum_{i=0}^{n-1} a(s,X(s)) \onei(s)\mathrm{d}s, \\
	B(t) & = \int\limits_0^t \sum_{i=0}^{n-1} b(s,X(s)) \onei(s)\mathrm{d}W(s).
\end{align*}
Let us denote by
$$ U_i = (t_i, \tilde X_n^{RM}(t_i)),\quad V_i = (\xi_i, \tilde X_n^{RM}(t_i)).
$$
Then, for the process $\{\tilde X_n^{RM} (t)\}_{t\in [0,T]}$ we can write
\begin{align*}
	\tilde X_n^{RM} (t) & = \eta + \tilde A_n^{RM} (t) + \tilde B_n^{RM}(t), \\
	\tilde A_n^{RM}(t) & = \int\limits_0^t \sum_{i=0}^{n-1} a(V_i) \onei(s)\mathrm{d}s, \\
	\tilde B_n^{RM}(t) & = \int\limits_0^t \sum_{i=0}^{n-1} \Bigl(b(U_i) + \int\limits_{t_i}^s L_1b(U_i)\mathrm{d}W(u)\Bigr)\onei(s)\mathrm{d}W(s).
\end{align*}
Note that the process
\begin{displaymath}
	\Bigl\{\sum_{i=0}^{n-1} \Bigl(b(U_i) + \int\limits_{t_i}^s L_1b(U_i)\mathrm{d}W(u)\Bigr)\onei(s)\Bigr\}_{s\in [0,T]}
\end{displaymath}
is adapted to $\{\tilde\Sigma^n_t\}_{t\in [0,T]}$ and has c\'adl\'ag paths. Hence, the It\^o integral above is well-defined.

We have that
\begin{displaymath}
	\mathbb{E} |A(t) - \tilde A_n^{RM} (t) |^q \leq C \sum_{k=1}^3 \mathbb{E}| \tilde A_{n,k}^{RM} (t)|^q,
\end{displaymath}
where
\begin{align*}
	\mathbb{E}| \tilde A_{n,1}^{RM} (t) |^q & = \mathbb{E} \Bigl| \int\limits_0^t \sum_{i=0}^{n-1}(a(s,X(s))-a(s, X(t_i))) \onei(s)\mathrm{d}s \Bigl|^q, \\ 
	\mathbb{E}| \tilde A_{n,2}^{RM} (t) |^q & = \mathbb{E} \Bigl| \int\limits_0^t \sum_{i=0}^{n-1}(a(s,X(t_i))-a(\xi_i, X(t_i))) \onei(s)\mathrm{d}s \Bigl|^q, \\ 
	\mathbb{E}| \tilde A_{n,3}^{RM} (t) |^q & = \mathbb{E} \Bigl| \int\limits_0^t \sum_{i=0}^{n-1}(a(\xi_i,X(t_i))-a(V_i)) \onei(s)\mathrm{d}s \Bigl|^q .
\end{align*}
It holds that
\begin{equation}
\label{est_AANRM_1}
 \mathbb{E}| \tilde A_{n,3}^{RM} (t)|^q \leq C \int\limits_0^t \sum_{i=0}^{n-1} \mathbb{E} |X(t_i) - \tilde X^{RM}_n(t_i)|^q \onei(s) \mathrm{d}s.
\end{equation}
From the It\^o formula we have
\begin{equation}
\label{Ito_form_f}
f(t, X(s)) - f(t, X(t_i)) = \int\limits_{t_i}^s \alpha(f,t,u)\mathrm{d}u + \int\limits_{t_i}^s \beta(f,t,u) \mathrm{d}W(u),
\end{equation}
where
\begin{eqnarray}
&&\alpha(f,t,u) = \frac{\partial f}{\partial y}(t,X(u))\cdot a(u,X(u))+\frac{1}{2}\frac{\partial^2 f}{\partial y^2}(t,X(u))\cdot b^2(u,X(u)),\\
&&\beta(f,t,u) = \frac{\partial f}{\partial y} (t, X(u)) \cdot b(u, X(u)),
\end{eqnarray}
for $f\in\{a,b\}$, $s,u\in [t_i,t_{i+1}]$, $t\in\{s,t_i\}$. By Lemma \ref{f_bounds} we get that
\begin{displaymath}
	|\alpha(f,t,u)|\leq C(1+|X(u)|^2),
\end{displaymath}
\begin{equation}
\label{beta_est}
	|\beta(f,t,u)|\leq C(1+|X(u)|),
\end{equation}
and
\begin{displaymath}
|\beta(a,t_1,u)-\beta(a,t_2,u)|\leq C(1+|X(u)|^2)\cdot |t_1-t_2|^{\gamma_1},
\end{displaymath}
for $t_1,t_2,u\in [t_i,t_{i+1}]$. Then, we can write that
\begin{displaymath}
	\mathbb{E}| \tilde A_{n,1}^{RM} (t) |^q \leq C \Bigl( \mathbb{E}|\tilde M^{RM}_{n,1} (t)|^q + \mathbb{E} |\tilde M^{RM}_{n,2}(t)|^q  \Bigr), 
\end{displaymath}
where
\begin{align*}	
	\mathbb{E}|\tilde M^{RM}_{n,1}(t)|^q&    = \mathbb{E} \Bigl| \int\limits_0^t \sum_{i=0}^{n-1} \Bigl( \int\limits_{t_i}^s \alpha(a,s,u)\mathrm{d}u \Bigr) \one_{[t_i, t_{i+1})}(s) \rd s\Bigr|^q \\
	\mathbb{E}| \tilde M^{RM}_{n,2}(t)|^q & = \mathbb{E} \Bigl| \int\limits_0^t \sum_{i=0}^{n-1} \Bigl( \int\limits_{t_i} ^s \beta(a,s,u) \mathrm{d}W(u) \Bigr) \one_{[t_i,t_{i+1})}(s) \mathrm{d}s \Bigr|^q.
\end{align*}
Note that for almost all $\omega\in\Omega$ the function
\begin{displaymath}
	[t_i,t_{i+1}]\times [t_i,t_{i+1}]\ni (s,u)\to \alpha(a,s,u)(\omega)\in\mathbb{R}
\end{displaymath}
is continuous. Hence, the parametric indefinite Riemann integral $\displaystyle{\Bigl\{\int\limits_{t_i}^s\alpha(a,s,u)\rd u\Bigr\}_{s\in [t_i,t_{i+1}]}}$ has almost all trajectories continuous.
Moreover, by \eqref{Ito_form_f} for all $s\in [t_i,t_{i+1}]$ it holds that
\begin{displaymath}
	\int\limits_{t_i}^s\beta(a,s,u)\mathrm{d}W(u) = a(s,X(s))-a(s,X(t_i))-\int\limits_{t_i}^s\alpha(a,s,u)\rd u.
\end{displaymath}
Thus the parametric indefinite stochastic It\^o integral $\displaystyle{\Bigl\{\int\limits_{t_i}^s\beta(a,s,u)\mathrm{d}W(u)\Bigr\}_{s\in [t_i,t_{i+1}]}}$ also has continuous modification. Thereby, $\mathbb{E}|\tilde M^{RM}_{n,1}(t)|^q$ and $\mathbb{E}|\tilde M^{RM}_{n,2}(t)|^q$ are well defined.

We have that
\begin{displaymath}
	\mathbb{E}|\tilde M^{RM}_{n,1}(t)|^q \leq T^{q-1} \sum_{i=0}^{n-1} \int\limits_{t_i}^{t_{i+1}} \Bigr((s-t_i)^{q-1}\cdot \int\limits_{t_i}^s \mathbb{E}|\alpha(a,s,u)|^q\mathrm{d}u\Bigl)\mathrm{d}s\leq Cn^{-q}. 
\end{displaymath}
Moreover, for any $t\in [0,T]$  there exists $l\in \{0, \ldots, n-1 \}$ such that $t\in[t_l, t_{l+1}]$ and
\begin{eqnarray}
\label{est_mn2_1}
	&&\mathbb{E}|\tilde M^{RM}_{n,2}(t)|^q\leq C\mathbb{E}\Bigl|\int\limits_{0}^{t_l}\sum\limits_{i=0}^{n-1}\Bigl(\int\limits_{t_i}^{s}(\beta(a,s,u)-\beta(a,t_i,u))\mathrm{d}W(u)\Bigr)\one_{[t_i, t_{i+1})}(s) \rd s\Bigl|^q\notag\\
	&&\quad\quad\quad\quad\quad +C\mathbb{E}\Bigl|\int\limits_{0}^{t_l}\sum\limits_{i=0}^{n-1}\Bigl(\int\limits_{t_i}^{s}\beta(a,t_i,u)\mathrm{d}W(u)\Bigr)\one_{[t_i, t_{i+1})}(s) \rd s\Bigl|^q\notag\\
	&&\quad\quad\quad\quad\quad +C\mathbb{E}\Bigl|\int\limits_{t_l}^t\Bigl(\int\limits_{t_l}^s\beta(a,s,u)\mathrm{d}W(u)\Bigr)\rd s\Bigl|^q.
\end{eqnarray}
By using the Burkholder and H\"older inequalities, together with Lemma \ref{moment_bound_sol}, we obtain
\begin{eqnarray}
\label{est_mn2_2}
	&&\mathbb{E}\Bigl|\int\limits_{0}^{t_l}\sum\limits_{i=0}^{n-1}\Bigl(\int\limits_{t_i}^{s}(\beta(a,s,u)-\beta(a,t_i,u))\mathrm{d}W(u)\Bigr)\one_{[t_i, t_{i+1})}(s) \rd s\Bigl|^q\notag\\
	&&\leq C\sum\limits_{i=0}^{n-1}\int\limits_{t_i}^{t_{i+1}}\Bigl((s-t_i)^{(q/2)-1}\cdot\mathbb{E}\int\limits_{t_i}^s|\beta(a,s,u)-\beta(a,t_i,u)|^q du\Bigr)\rd s\notag\\
	&&\leq C\sum\limits_{i=0}^{n-1}\int\limits_{t_i}^{t_{i+1}}\Bigl((s-t_i)^{(q/2)-1}\cdot\int\limits_{t_i}^s(1+\mathbb{E}|X(u)|^{2q})\cdot (s-t_i)^{q\gamma_1} \rd u\Bigr) \rd s \notag\\
&&\leq Cn^{-q(\frac{1}{2}+\gamma_1)},
\end{eqnarray}
and
\begin{eqnarray}
\label{est_double_Ito_1}
	&&\mathbb{E}\Bigl|\int\limits_{t_l}^t\Bigl(\int\limits_{t_l}^s\beta(a,s,u)\mathrm{d}W(u)\Bigr)\rd s\Bigl|^q\notag\\
	&&\leq C\Bigl(\frac{T}{n}\Bigr)^{q-1}\cdot\int\limits_{t_l}^{t_{l+1}}\Bigl((s-t_l)^{(q/2)-1}\cdot\mathbb{E}\int\limits_{t_l}^s|\beta(a,s,u)|^q\rd u\Bigr)\rd s\notag\\
	&&\leq Cn^{-q+1}\int\limits_{t_l}^{t_{l+1}}\Bigl((s-t_l)^{(q/2)-1}\cdot\int\limits_{t_l}^s(1+\mathbb{E}|X(u)|^q)\rd u\Bigr)\rd s\leq Cn^{-3q/2}.
\end{eqnarray}
Let 
\begin{displaymath}
 Y_i = \int\limits_{t_i}^{t_{i+1}}\Bigl( \int\limits_{t_i}^s \beta(a,t_i,u)\mathrm{d}W(u) \Bigr)\mathrm{d}s \quad{i=0,1,\ldots,n-1},
\end{displaymath}
and
\begin{displaymath}
 Z_k = \sum_{i=0}^k Y_i \qquad k=0,1,\ldots,n-1,
\end{displaymath} 
where $Z_{-1}:=0$. Therefore,
\begin{equation}
\label{est_mn2_3}
	\mathbb{E}\Bigl|\int\limits_{0}^{t_l}\sum\limits_{i=0}^{n-1}\Bigl(\int\limits_{t_i}^{s}\beta(a,t_i,u)\mathrm{d}W(u)\Bigr)\one_{[t_i, t_{i+1})}(s) \rd s\Bigl|^q=\mathbb{E}|Z_{l-1}|^q, \ l \in\{0,1,\ldots,n-1\}.
\end{equation}
Notice that the process $\displaystyle{\Bigl\{\int\limits_{t_i}^s \beta(a,t_i,u)\mathrm{d}W(u)\Bigr\}_{s\in [t_i,t_{i+1}]}}$ 
is adapted to the filtration $\{\Sigma_s\}_{s\in [t_i,t_{i+1}]}$ and has continuous paths. Hence, it is progressively measurable. This and Fubini theorem imply that $Y_i$ is $\Sigma_{t_{i+1}}$-measurable. Furthermore, let $ \mathcal{G}_i := \Sigma_{t_{i+1}}$, $i\in\{0,1,\ldots,n-1\}$.
Then $\{\mathcal{G}_i\}_{i\in\{0,1,\ldots,n-1\}}$ is a filtration and $Z_k$ is $\mathcal{G}_k$ measurable for each $k=0,1,\ldots,n-1$. 
By using the Fubini theorem for conditional expectation (see, for example, \cite{Brks}) and martingale property of It\^o integral  we have  
\begin{displaymath}
	\mathbb{E}(Z_{k+1}-Z_k | \mathcal{G}_k) = \int\limits_{t_{k+1}}^{t_{k+2}} \mathbb{E} \Biggl(\int\limits_{t_{k+1}}^s \beta(a,t_{k+1},u) \mathrm{d}W(u)\bigg|\mathcal{G}_k\Biggr) \mathrm{d}s = 0,
\end{displaymath}
for $k=0,1,\ldots,n-2$.  This implies that $\{Z_k, \mathcal{G}_k\}_{k\in\{0, 1, \ldots, n-1\}}$ is a discrete-time martingale. Therefore, by  using the discrete version of the Burkholder inequality we have for every $k \in \{0,1,\ldots,n-1\}$ that
\begin{displaymath}
 \mathbb{E} |Z_k|^q \leq C_q^q \mathbb{E} \Bigl(\sum_{i=0}^k Y_i^2\Bigr)^{q/2} \leq C_q^q n^{q/2 -1} \sum_{i=0}^{n-1} \mathbb{E} |Y_i|^q.
\end{displaymath}
Moreover, analogously as in \eqref{est_double_Ito_1} we get that
\begin{displaymath}
	\mathbb{E}|Y_i|^q\leq Cn^{-3q/2}, \quad i = 0,1,\ldots,n-1.
\end{displaymath}
Therefore, for any $k = 0, 1, \ldots, n-1$
\begin{equation}
\label{est_mn2_4}
	\mathbb{E}|Z_k|^q \leq C n^{-q}.
\end{equation} 
Combining together \eqref{est_mn2_1}, \eqref{est_mn2_2}, \eqref{est_mn2_3} and \eqref{est_mn2_4}, we have
$$ \mathbb{E}| \tilde M_{n,2}^{RM}(t)| \leq  C n^{-q\min\{\frac{1}{2}+\gamma_1,1\}}.
$$
Therefore, for any $t\in[0,T]$ 
\begin{equation}
\label{est_AARM_2}
 \mathbb{E}| \tilde A^{RM}_{n,1}(t)|^q \leq C n^{-q\min\{\frac{1}{2}+\gamma_1,1\}}.
\end{equation}
We now bound from above $\sup\limits_{t\in [0,T]}\mathbb{E}|\tilde A^{RM}_{n,2}(t)|^q$. The estimation goes analogously as in \cite{PrMo14}, with some minor adjustments needed in order to include the H\"older regularity. For convenience of the reader we present a complete estimation procedure.

We denote by
\begin{displaymath}
	i(t)=\sup\{i=0,1,\ldots,n \ | \ iT/n\leq t\}, 
\end{displaymath}
\begin{displaymath}
	\zeta(t)=i(t)\frac{T}{n},
\end{displaymath}
for $t\in [0,T]$. Now we can write that 
\begin{equation}
	\label{M2_EST_1}
	\mathbb{E}|\tilde A^{RM}_{n,2}(t)|^q\leq 2^{q-1}\Bigl(\mathbb{E}|\tilde A^{RM}_{n,21}(t)|^q+\mathbb{E}|\tilde A^{RM}_{n,22}(t)|^q\Bigr),
\end{equation}
with
\begin{eqnarray}
	&&\mathbb{E}|\tilde A^{RM}_{n,21}(t)|^q=\mathbb{E}\Bigl|\sum_{k=0}^{i(t)-1}\int\limits_{t_k}^{t_{k+1}}\Bigl(a(s,X(t_k))-a(\xi_k,X(t_k))\Bigr)\rd s\Bigl|^q,\\
	&&\mathbb{E}|\tilde A^{RM}_{n,22}(t)|^q=\mathbb{E}\Bigl|\;\int\limits_{\zeta(t)}^t\Bigl(a(s,X(\zeta(t)))-a(\xi_{i(t)},X(\zeta(t)))\Bigr)\rd s\Bigl|^q,
\end{eqnarray}
for all $t\in [0,T]$, where we take $\mathbb{E}|\tilde A^{RM}_{n,22}(T)|^q=0$. Moreover, let
\begin{equation}
	\label{DEF_T_Y_K}
	\tilde Y_k=\int\limits_{t_k}^{t_{k+1}}\Bigl(a(s,X(t_k))-a(\xi_k,X(t_k))\Bigr)\rd s, \ k=0,1,\ldots,n-1,
\end{equation}
and 
\begin{displaymath}
	\tilde Z_j=\sum_{k=0}^j\tilde Y_k, \ j=0,1,\ldots,n-1,
\end{displaymath}
where we set $\tilde Z_{-1}:=0$. Note that
\begin{displaymath}
	|\tilde Y_k|\leq K  \ \Bigl(1+\sup\limits_{0\leq t \leq T}|X(t)|\Bigr) \ (T/n)^{\gamma_1+1},
\end{displaymath}
and conditioned on $\Sigma_{\infty}$ the random variables $(\tilde Y_k)_{k=0}^{n-1}$ are zero mean, independent, and bounded by $\displaystyle{K\Bigl(1+\sup\limits_{0\leq t \leq T}|X(t)|\Bigr)(T/n)^{\gamma_1+1}}$. Therefore, by applying Theorem 4 from \cite{JenNeuen} and Lemma \ref{moment_bound_sol} we have for all $t\in [0,T]$ that
\begin{eqnarray}
	\label{TM_1_EST_4}
\mathbb{E}|\tilde A^{RM}_{n,21}(t)|^q&=&\mathbb{E}|\tilde Z_{i(t)-1}|^q \leq \mathbb{E}\Biggl[\mathbb{E}\Biggl(\max\limits_{0\leq j\leq n-1}|\tilde Z_j|^q \ | \ \Sigma_{\infty}\Biggr)\Biggr]\notag\\
&\leq& C_2 (T/n)^{q(\gamma_1+1)}\cdot n^{q/2}\cdot\mathbb{E}\Bigl(1+\sup\limits_{t\in [0,T]}|X(t)|\Bigr)^q\leq C_3 n^{-q(\gamma_1+\frac{1}{2})},
\end{eqnarray}
where $C_2, C_3>0$ depend only on the parameters of the class $\mathcal{F}(\gamma_1,\gamma_2,q,K)$ and $q$. Moreover, due to the fact that $\Sigma_{\infty}$ and $\sigma(\xi_0,\xi_1,\ldots,\xi_{n-1})$ are independent $\sigma$-fields, we get for all $t\in [0,T)$ that
\begin{eqnarray*}
	\mathbb{E}|\tilde A^{RM}_{n,22}(t)|^q &\leq& (t-\zeta(t))^{q-1}\cdot\int\limits_{\zeta(t)}^{t}\mathbb{E}|a(s,X(\zeta(t)))-a(\xi_{i(t)},X(\zeta(t)))|^q \rd s\notag\\
	&\leq& C_4(t-\zeta(t))^{q-1}\cdot\Bigl(1+\sup\limits_{t\in [0,T]}\mathbb{E}|X(t)|^q\Bigr)\cdot\mathbb{E}\int\limits_{\zeta(t)}^t |s-\xi_{i(t)}|^{q\gamma_1}\rd s,
\end{eqnarray*}
and $\mathbb{E}|\tilde A^{RM}_{n,22}(T)|^q=0$.
Note that for $t\in [0,T)=\bigcup\limits_{i=0}^{n-1}[t_i,t_{i+1})$ we have that $\zeta(t)\leq t<\zeta(t)+h$. In addition, for $t\in [0,T)$ we have that $\xi_{i(t)}$ is uniformly distributed on $[\zeta(t),\zeta(t)+h]$. Hence, $|s-\xi_{i(t)}|\leq h$ for all $t\in [0,T)$ and $s\in [\zeta(t),t]\subset [\zeta(t),\zeta(t)+h)$, which gives for all $t\in [0,T)$ 
\begin{displaymath}
	\mathbb{E}\int\limits_{\zeta(t)}^t |s-\xi_{i(t)}|^{q\gamma_1}\rd s \leq (t-\zeta(t))\cdot h^{q\gamma_1}.
\end{displaymath}
Therefore,
\begin{equation}
\label{TM_2_EST1}
	\mathbb{E}|\tilde A^{RM}_{n,22}(t)|^q\leq C_5 n^{-q(1+\gamma_1)}.
\end{equation}
Using \eqref{M2_EST_1}, \eqref{TM_1_EST_4} and \eqref{TM_2_EST1} we obtain
\begin{equation}
	\label{M2_EST_2}
	\mathbb{E}|\tilde A^{RM}_{n,2}(t)|^q\leq C_6 n^{-q(\gamma_1+\frac{1}{2})},
\end{equation}
for all $t\in [0,T]$. Combining \eqref{est_AANRM_1}, \eqref{est_AARM_2} and \eqref{M2_EST_2} we get
\begin{equation}	
\label{est_AAN}
	\mathbb{E}|A(t)-\tilde A_n^{RM}(t)|^q\leq C_1\int\limits_0^t\sum\limits_{i=0}^{n-1}\mathbb{E}|X(t_i)-\tilde X^{RM}_n(t_i)|^q\onei(s)\rd s+C_2 n^{-q\min\{\frac{1}{2}+\gamma_1,1\}}.
\end{equation}

The analysis of the diffusion part is as follows. For all $t\in [0,T]$
\begin{displaymath}
 \mathbb{E}|B(t) - \tilde B_n^{RM}(t)|^q \leq C \sum_{k=1}^3 \mathbb{E}|\tilde B_{n,k}^{RM}(t)|^q,
\end{displaymath}
where
\begin{eqnarray}
	&&\mathbb{E} |\tilde B_{n,1}^{RM} (t)|^ q = \mathbb{E} \Bigl| \int\limits_0^t \sum_{i=0}^{n-1} \left( b(s,X(s)) - b(t_i, X(s))\right) \one_{[t_i, t_{i+1})}(s)\mathrm{d}W(s)\Bigl|^q, \\
	\label{est_BN2}
	&&\mathbb{E}|\tilde B_{n,2}^{RM} (t)|^ q = \mathbb{E} \Bigl| \int\limits_0^t \sum_{i=0}^{n-1} \Bigl( b(t_i, X(s)) - b(t_i, X(t_i))\notag \\
&&\quad\quad\quad\quad\quad\quad\quad\quad -\int\limits_{t_i}^s L_1 b(U_i)\mathrm{d}W(u)\Bigr)\one_{[t_i, t_{i+1})}(s)\mathrm{d}W(s)\Bigl|^q, \\
	&&\mathbb{E}| \tilde B_{n,3}^{RM} (t)|^q = \mathbb{E} \Bigl| \int\limits_0^t \sum_{i=0}^{n-1} \left( b(t_i, X(t_i)) - b(U_i)\right) \one_{[t_i, t_{i+1})}(s)\mathrm{d}W(s)\Bigl|^q.
\end{eqnarray}
By the Burkholder inequality and Lemma \ref{moment_bound_sol} we have for every $t\in[0,T]$ that
\begin{align*}
	\mathbb{E} |\tilde B_{n,1}^{RM}(t_i)|^q 
	 &\leq C \int\limits_0^T \mathbb{E} \sum_{i=0}^{n-1} |b(s,X(s)) - b(t_i, X(s))|^q \one_{[t_i,t_{i+1})}(s)\mathrm{d}s \\
	 &\leq C \sum\limits_{i=0}^{n-1} \int\limits_{t_i}^{t_{i+1}} \mathbb{E}(1+|X(s)|)^q\cdot (s-t_i)^{q\gamma_2} \mathrm{d}s\leq C n^{-q\gamma_2},
\end{align*}
and
\begin{displaymath}
	\mathbb{E}|\tilde B_{n,3}^{RM} (t)|^q \leq C_2 \int\limits_0^t \sum_{i=0}^{n-1} \mathbb{E}|X(t_i) - \tilde X_n^{RM}(t_i)|^q \one_{[t_i,t_{i+1})}(s)\mathrm{d}s.
\end{displaymath}
From \eqref{Ito_form_f} we get for $s\in[t_i, t_{i+1}]$ that
\begin{eqnarray}
\label{ito_f2}
	&& b(t_i, X(s)) - b(t_i, X(t_i)) - \int\limits_{t_i}^s L_1 b(U_i)\mathrm{d}W(u)\notag \\
	&& = \int\limits_{t_i}^s \alpha(b,t_i,u)\mathrm{d}u + \int\limits_{t_i}^s \Bigl( \beta(b,t_i,u) - L_1 b(U_i) \Bigr) \mathrm{d}W(u).
\end{eqnarray}
Hence, from \eqref{est_BN2}, \eqref{ito_f2},  and by the Burkholder inequality we get
\begin{eqnarray}
	&&\mathbb{E}|\tilde B_{n,2}^{RM}(t)|^q \leq C_1 \mathbb{E} \int\limits_0^t \sum_{i=0}^{n-1} \Bigl| \int\limits_{t_i}^s \alpha(b,t_i,u)\mathrm{d}u \Bigl|^q \one_{[t_i,t_{i+1})}(s)\rd s\notag \\
	\label{est_BRM_21}
	&&+ C_2 \mathbb{E} \int\limits_0^t \sum_{i=0}^{n-1} \Bigl| \int\limits_{t_i}^s \left( \beta(b,t_i,u) - L_1b(U_i) \right) \mathrm{d}W(u)\Bigl|^q \one_{[t_i,t_{i+1})}(s) \mathrm{d}s,
\end{eqnarray}
where 
\begin{eqnarray}
\label{est_BRM_22}
&&\mathbb{E} \int\limits_0^t \sum_{i=0}^{n-1} \Bigl| \int\limits_{t_i}^s \alpha(b,t_i,u)\mathrm{d}u \Bigl|^q \one_{[t_i,t_{i+1})}(s) \mathrm{d}s
 \leq \sum_{i=0}^{n-1}\int\limits_{t_i}^{t_{i+1}} (s-t_i)^{q-1} \int\limits_{t_i}^s \mathbb{E} |\alpha(b,t_i,u)|^q \mathrm{d}u \mathrm{d}s\notag \\
&& \leq C \cdot \Bigl( 1 + \sup_{t\in[0,T]} \mathbb{E}|X(t)|^{2q} \Bigr)\cdot \sum_{i=0}^{n-1} \int\limits_{t_i}^{t_{i+1}} (s-t_i)^q \mathrm{d}s\leq C n^{-q},
\end{eqnarray}
\begin{eqnarray}
\label{est_BRM_23}
&&\mathbb{E} \int\limits_0^t \sum_{i=0}^{n-1} \Bigl| \int\limits_{t_i}^s \left( \beta(b,t_i,u) - L_1b(U_1) \right) \mathrm{d}W(u)\Bigl|^q \one_{[t_i,t_{i+1})}(s) \mathrm{d}s\notag \\
&& \leq C\cdot n^{-(q/2)+1}\cdot\int\limits_0^t \sum_{i=0}^{n-1}\Bigl( \int\limits_{t_i}^s \mathbb{E} |\beta(b,t_i,u)-L_1b(U_i)|^q \mathrm{d}u\Bigr) \one_{[t_i, t_{i+1})}(s)\mathrm{d}s.
\end{eqnarray}
Note that
\begin{eqnarray}
	&&|\beta(b,t_i,u) - L_1 b(U_i)| \leq K^2(1+|X(u)|)\cdot |u-t_i|^{\gamma_2}+K|X(u)-X(t_i)|\notag\\
&&\quad\quad\quad\quad\quad+K|X(t_i)-\tilde X^{RM}_n(t_i)|. 
\end{eqnarray}
and, therefore, for any $s\in[t_i, t_{i+1}]$ we have
\begin{equation}
\label{est_BRM_24}
	\int\limits_{t_i}^s \mathbb{E}|\beta(b,t_i,u) - L_1b(U_i)|^q\mathrm{d}u \leq  \tilde C_1 n^{-q\gamma_2-1} + C_2n^{-(q/2)-1}+ C_3 n^{-1}\mathbb{E}| X(t_i) - \tilde X_n^{RM}(t_i))|^q.
\end{equation}
From \eqref{est_BRM_21}, \eqref{est_BRM_22}, \eqref{est_BRM_23} and \eqref{est_BRM_24} we obtain that
\begin{align*}
	\mathbb{E}|\tilde B_{n,2}^{RM} (t)|^q \leq & C_1 n^{-q} +C_2 n^{-q(\frac{1}{2}+\gamma_2)} +  C_3 \int\limits_0^t \sum_{i=0}^{n-1} \mathbb{E}|X(t_i) - \tilde X_n^{RM}(t_i)|^q \one_{[t_i,t_{i+1})}(s)\mathrm{d}s.
\end{align*}
Hence, for any $t\in[0,T]$ we have
\begin{equation}
	\label{est_BBN}
	\mathbb{E}|B(t)-\tilde B_n^{RM}(t)|^q \leq K_1 n^{-q\gamma_2}+ K_2 \int\limits_0^t\sum_{i=0}^{n-1} \mathbb{E}| X(t_i) - \tilde X_n^{RM}(t_i)|^q \one_{[t_i, t_{i+1})}(s) \mathrm{d}s.
\end{equation}
By \eqref{est_AAN} and \eqref{est_BBN} we get for all $t\in[0,T]$ that  
\begin{displaymath}
	\mathbb{E}|X(t) - \tilde X_n^{RM}(t)|^q \leq C_1 n^{-q\min\{\frac{1}{2}+\gamma_1,\gamma_2\}} +C_2\int\limits_0^t\sum_{i=0}^{n-1} \mathbb{E} \left| 	X(t_i) - \tilde X_n^{RM}(t_i) \right|^q \one_{[t_i, t_{i+1})}(s)\mathrm{d}s,
\end{displaymath}
which implies for all $t\in [0,T]$ that
\begin{displaymath}
	\sup\limits_{0\leq s\leq t}\mathbb{E}|X(s) - \tilde X_n^{RM}(s)|^q\leq C_1 n^{-q\min\{\frac{1}{2}+\gamma_1,\gamma_2\}}+C_2\int\limits_0^t \sup\limits_{0\leq u\leq s}\mathbb{E} | X(u) - \tilde X_n^{RM}(u)|^q \mathrm{d}s.
\end{displaymath}
Finally, by using the Gronwall's inequality we arrive at \eqref{est_XXRM},
%\begin{equation}
%\label{cont_RM_est_1}
 % \sup_{t\in[0,T]} \| X(t) - \tilde X_n^{RM}(t) \|_q \leq Cn^{-\min\{\frac{1}{2}+\gamma_1,\gamma_2\}},
% \end{equation}
 which ends the proof. \ \ \ $\blacksquare$
 %%%%%%%%%%%%%%%%%%%%%%%%%
 \begin{remark}
 In particular, if $\gamma_2=\min\{\frac{1}{2}+\gamma_1,1\}$  then
 \begin{displaymath}
 \sup_{t\in[0,T]} \| X(t) - \tilde X_n^{RM}(t) \|_q \leq Cn^{-\min\{\frac{1}{2}+\gamma_1,1\}},
 \end{displaymath}
 which recovers the analogous result from \cite{KRWU}.
 \end{remark}
%%%%%%%%%%%%%%
\begin{remark} We compare the errors of the classical Euler method $\mathcal{A}^E_n$, randomized Euler algorithm $\mathcal{A}^{RE}_n$, classical Milstein scheme $\mathcal{A}^M_n$, and randomized Milstein algorithm $\mathcal{A}^{RM}_n$ in the class $\mathcal{F}(\gamma_1,\gamma_2,q,K)$. Namely, in the case of exact information about $a$ and $b$, we have that
\begin{eqnarray}
	&&e^{(q)} (\mathcal{A}^E_n, \mathcal{F} (\gamma_1,\gamma_2,q,K), W, V^i, 0,0)=O(n^{-\min\{\gamma_1,\gamma_2,1/2\}}),\notag\\
	&&e^{(q)} (\mathcal{A}^{RE}_n, \mathcal{F} (\gamma_1,\gamma_2,q,K), W, V^i, 0,0)=O(n^{-\min\{1/2,\gamma_2\}}),\notag\\
	&&e^{(q)} (\mathcal{A}^M_n, \mathcal{F} (\gamma_1,\gamma_2,q,K), W, V^i, 0,0)=O(n^{-\min\{\gamma_1,\gamma_2\}}),\notag\\
	&&e^{(q)} (\mathcal{A}^{RM}_n, \mathcal{F} (\gamma_1,\gamma_2,q,K), W, V^i, 0,0)=O(n^{-\min\{\frac{1}{2}+\gamma_1,\gamma_2\}}).
\end{eqnarray}
Hence, if $\gamma\in (0,1/2]$ and $\gamma_2\in (0,1]$ then $\mathcal{A}^E_n$ and $\mathcal{A}^M_n$ have the same error $O(n^{-\min\{\gamma_1,\gamma_2\}})$. Moreover, for $\gamma_1\in (0,1]$ and $\gamma_2\in (0,1/2]$ the methods $\mathcal{A}^{RE}_n$ and $\mathcal{A}^{RM}_n$ have the same error $O(n^{-\gamma_2})$. Finally, for $\gamma_1\in(1/2,1)$ and $\gamma_2\in (1/2,1]$ the randomized Milstein algorithm $\mathcal{A}^{RM}_n$ outperforms $\mathcal{A}^{E}_n$, $\mathcal{A}^{RE}_n$, and $\mathcal{A}^{M}_n$.
\end{remark}
%%%%%%%%%%%%%%%%%%%%%%%%%%%%%%
\subsection{Performance of randomized derivative-free Milstein algorithm for exact information} \label{sec:Perf_randdfMilstein}
\noindent\newline
In this section we analyze the error of the algorithm $\mathcal{A}^{df-RM}_n$ in the case of exact information. Recall that its time-continuous version is denoted by $\tilde X_n^{df-RM}=\{\tilde X_n^{df-RM}(t)\}_{t\in [0,T]}$. 

We now give a proof of the following results. 
\begin{proposition}
	\label{PROP_RMDF_ERR} 
There exists a positive constant $C$, depending only on the parameters of the class $\mathcal{F}(\gamma_1,\gamma_2,q,K)$, such that for all $n\in\mathbb{N}$ and all $(a,b,\eta)\in\mathcal{F}(\gamma_1,\gamma_2,q,K)$ we have
\begin{equation}
\label{est_XXDFRM}
	\sup\limits_{t\in [0,T]}\|X(t)-\tilde X_n^{df-RM} (t)\|_q\leq C n^{-\min\{\frac{1}{2}+\gamma_1,\gamma_2\}},
\end{equation}
and, in particular,
	\begin{displaymath}
		\|X(T)-\mathcal{A}_n^{df-RM}(a,b,\eta,W,0,0)\|_q \leq C n^{-\min\{\frac{1}{2}+\gamma_1,\gamma_2\}}.
	\end{displaymath}
\end{proposition}
{\bf Proof.} By \eqref{est_XXRM} we have that
\begin{equation}
\label{tineq_est_rmdf}
\sup\limits_{t\in [0,T]} \|X(t)-\tilde {X}_n^{df-RM}(t)\|_q\leq Cn^{-\min\{\frac{1}{2}+\gamma_1,\gamma_2\}}+\sup\limits_{t\in [0,T]} \|\tilde {X}_n^{RM}(t)-\tilde {X}_n^{df-RM}(t)\|_q.
\end{equation}
Hence, we only need to estimate
 $$ \sup_{t\in[0,T]} \| \tilde X_n^{RM}(t) - \tilde {X}_n^{df-RM}(t) \|_q.
$$
Recall that
$$ U_i = (t_i, \tilde X_n^{RM}(t_i)),\quad V_i = (\xi_i, \tilde X_n^{RM}(t_i)).
$$
In addition, let us denote by
$$ U_i^{df} = (t_i, \tilde X_n^{df-RM}(t_i)),\quad V_i^{df} = (\xi_i, \tilde X_n^{df-RM}(t_i)).
$$
We have that for all $t\in [0,T]$
\begin{align*}
	\xndfc (t) & = \eta + \tilde A_n^{df-RM} (t) + \tilde B_n^{df-RM}(t), \\
	\tilde A_n^{df-RM}(t) & = \int\limits_0^t \sum_{i=0}^{n-1} a(V_i^{df}) \onei(s)\mathrm{d}s, \\
	\tilde B_n^{df-RM}(t) & = \int\limits_0^t \sum_{i=0}^{n-1} \Bigl(b(U_i^{df}) + \int\limits_{t_i}^s \mathcal{L}_{1,h}b(U_i^{df})\mathrm{d}W(u)\Bigr)\onei(s)\mathrm{d}W(s).
\end{align*}
Then
\begin{equation}
\label{est_ADFRM}
\mathbb{E}|\tilde A_n^{RM}(t) - \tilde A_n^{df-RM}(t)|^q \leq C_1 \int\limits_0^t \sum_{i=0}^{n-1} \mathbb{E}|\tilde X_n^{RM}(t_i) - \tilde X_n^{df-RM}(t_i)|^q \one_{[t_i, t_{i+1})}(s)\mathrm{d}s,
\end{equation}
and
%\begin{multline*}
% \tilde B_n^{RM}(t) - \tilde B_n^{df-RM}(t) = \int\limits_0^t \sum_{i=0}^{n-1} \bigg( (b(U_i) - b(U_i^{df}))  \\
% + \int\limits_{t_i}^s \left(L_1b(U_i)-\mathcal{L}_{1,h}b(U_i^{df})\right)\mathrm{d}W(u)\bigg)\one_{[t_i, t_{i+1})}(s)\mathrm{d}W(s).
%\end{multline*}
%Therefore,
\begin{align*}
	\mathbb{E}|\tilde B_n^{RM}(t) - \tilde B_n^{df-RM}(t)|^q \leq & C \mathbb{E}\Bigl| \int\limits_0^t\sum_{i=0}^{n-1} (b(U_i)-b(U_i^{df}))\onei(s)\rd W(s)\Bigl|^q \\
	& + C \mathbb{E} \Bigl| \int\limits_0^t \sum_{i=0}^{n-1}\Bigl(\int\limits_{t_i}^s \left( L_1b(U_i) - \mathcal{L}_{1,h}b(U_i^{df}) \right) \rd W(u)\Bigr) \onei(s)\rd W(s)\Bigl|^q.
\end{align*}
Furthermore, by the Burkholder inequality and Lemma \ref{DIV_DIFF_PROP1}
\begin{displaymath}
\mathbb{E}\Bigl|\int\limits_0^t\sum_{i=0}^{n-1} (b(U_i)-b(U_i^{df}))\onei(s)\rd W(s)\Bigl|^q \leq C \int\limits_0^t\sum_{i=0}^{n-1} \mathbb{E}|\tilde X_n^{RM}(t_i) -\tilde X_n^{df-RM}(t_i)|^q \onei(s)\rd s,
\end{displaymath}
\begin{align*}
	& \mathbb{E}\Bigl| \int\limits_0^t \sum_{i=0}^{n-1}\Bigl( \int\limits_{t_i}^s \left( L_1b(U_i) - \mathcal{L}_{1,h}b(U_i^{df})\right)\rd W(u)\Bigr) \onei(s) \rd W(s)\Bigl|^q \\
	& \leq C \int\limits_0^t \sum_{i=0}^{n-1}\Bigl((s-t_i)^{q/2-1}\cdot \int\limits_{t_i}^s \mathbb{E}|L_1b(U_i) - \mathcal{L}_{1,h}b(U_i^{df})|^q \rd u \Bigr) \onei (s)\rd s\\
	&\leq C_1 \int\limits_0^t \sum_{i=0}^{n-1} \mathbb{E}|\tilde X_n^{RM}(t_i) - \tilde X_n^{df-RM}(t_i)|^q \onei(s)\rd s+C_2 n^{-3q/2}\Bigl(1+\sup_{t\in[0,T]} \mathbb{E}|\tilde X_n^{df-RM}(t)|^q \Bigr).
\end{align*}
Therefore,
\begin{align}
\label{est_BBDF_RM}
	\mathbb{E}|\tilde B_n^{RM}(t) - \tilde B_n^{df-RM}(t)|^q \leq & C_1 \int\limits_0^t\sum_{i=0}^{n-1}\mathbb{E}|\tilde X_n^{RM}(t_i) - \tilde X_n^{df-RM}(t_i)|^q \onei(s)\rd s\notag \\
	& + C_2 \Bigl(1+\sup_{t\in[0,T]} \mathbb{E}|\tilde X_n^{df-RM}(t)|^q \Bigr) n^{-3q/2}.
\end{align}
Hence, from \eqref{est_ADFRM} and \eqref{est_BBDF_RM} we get for all $t\in [0,T]$
\begin{align*}
	\mathbb{E}|\tilde X_n^{RM}(t) - \tilde X_n^{df-RM}(t)|^q \leq & C_1 \int\limits_0^t\sum_{i=0}^{n-1}\mathbb{E}|\tilde X_n^{RM}(t_i) - \tilde X_n^{df-RM}(t_i)|^q \onei(s)\rd s \\
	& + C_2 \Bigl(1+\sup_{t\in[0,T]} \mathbb{E}|\tilde X_n^{df-RM}(t)|^q \Bigr) n^{-3q/2}.
\end{align*}
Hence, by the Gronwall's lemma we obtain
\begin{equation}
\label{tineq_est_rmdf2}
	\mathbb{E}|\tilde X_n^{RM}(t) - \tilde X_n^{df-RM}(t)|^q \leq C\Bigl(1+\sup_{t\in[0,T]} \mathbb{E}|\tilde X_n^{df-RM}(t)|^q \Bigr) n^{-3q/2}.
\end{equation}
Therefore, by \eqref{tineq_est_rmdf}, \eqref{tineq_est_rmdf2} and Lemma \ref{BoundRanMilNo} 
\begin{eqnarray}
\label{tineq_est_rmdf3}
	&&\sup\limits_{0\leq t \leq T}\|  X(t) - \tilde X_n^{df-RM}(t) \|_q \leq C_1\Bigl(1+\sup_{t\in[0,T]} \| \tilde X^{df-RM}_n(t) \|_q \Bigr) n^{-3/2}\notag\\
	&&\quad\quad\quad\quad+C_2 n^{-\min\{\frac{1}{2}+\gamma_1,\gamma_2\}}\leq  Cn^{-\min\{\frac{1}{2}+\gamma_1,\gamma_2\}}.
\end{eqnarray}
This ends the proof. \ \ \ $\blacksquare$
\\ \\
%%%%%%%%%%
%%%%%%%%%%
%%%%%%%%%%
Having Proposition \ref{PROP_RMDF_ERR} we are ready to prove Theorem \ref{RMDF_NOIZZ_ERR}.
%%%
\subsection{Proof of Theorem \ref{RMDF_NOIZZ_ERR}}
\noindent\newline
We set
\begin{displaymath}
 {\bar U}_i^{df} = (t_i,  \xndfnc(t_i)),\qquad \bar V_i^{df} = (\xi_i, \xndfnc(t_i)).
\end{displaymath}
The process $\{\xndfnc (t)\}_{t\in [0,T]}$ can be decomposed as follows
\begin{align*}
	\xndfnc (t) & = \eta + {\tilde {\bar A}_n^{df-RM}}(t) + {\tilde {\bar B}_n^{df-RM}}(t),
\end{align*}
where
\begin{eqnarray}
	&&{\tilde {\bar A}_n^{df-RM}}(t)  = \int\limits_0^t \sum_{i=0}^{n-1} \tilde a(\bar V_i^{df}) \onei(s)\mathrm{d}s, \notag\\
	\label{ito_int_BRM}
	&&{\tilde {\bar B}_n^{df-RM}}(t)  = \int\limits_0^t \sum_{i=0}^{n-1} \Bigl(\tilde b({\bar U}_i^{df}) + \int\limits_{t_i}^s \mathcal{L}_{1,h}\tilde b({\bar U}_i^{df})\mathrm{d}W(u)\Bigr)\onei(s)\mathrm{d}W(s).
\end{eqnarray}
Due to Lemma \ref{PROG_RMDFNO_PROC} the process
\begin{displaymath}
	\Bigl\{\sum_{i=0}^{n-1} \Bigl(\tilde b({\bar U}_i^{df}) + \int\limits_{t_i}^s \mathcal{L}_{1,h}\tilde b({\bar U}_i^{df})\mathrm{d}W(u)\Bigr)\onei(s)\Bigr\}_{s\in [0,T]}
\end{displaymath}
is adapted to $\{\tilde\Sigma^n_t\}_{t\in [0,T]}$ and has c\'adl\'ag paths. Hence, the It\^o integral in \eqref{ito_int_BRM} is well-defined.

From \eqref{est_XXDFRM} we have that 
\begin{equation}
\label{tineq_est_rmdf4}
	\sup\limits_{t\in [0,T]}\|X(t)-\xndfnc (t)\|_q\leq C n^{-\min\{\frac{1}{2}+\gamma_1,\gamma_2\}}+\sup\limits_{t\in [0,T]}\|\tilde X_n^{df-RM}(t)-\xndfnc (t)\|_q,
\end{equation}
and we only need to estimate $\sup\limits_{t\in [0,T]}\|\tilde X_n^{df-RM}(t)-\xndfnc (t)\|_q$.
%\begin{displaymath}
%\tilde X_n^{df-RM}(t)-\xndfnc (t)=({\tilde {A}_n^{df-RM}}(t)-{\tilde {\bar A}_n^{df-RM}}(t))+({\tilde {B}_n^{df-RM}}(t)-{\tilde {\bar B}_n^{df-RM}}(t)).
%\end{displaymath}
We have that
\begin{eqnarray}
\label{est_TANADFRM}
 &&\mathbb{E}|{\tilde {A}_n^{df-RM}}(t)-{\tilde {\bar A}_n^{df-RM}}(t)|^q \leq C \int\limits_0^t \sum_{i=0}^{n-1} \mathbb{E}|\tilde X^{df-RM}_n(t_i)-\xndfnc (t_i)|^q\one_{[t_i, t_{i+1})}(s)\mathrm{d}s\notag\\
 &&\quad\quad\quad\quad+C\Bigl(1+\sup\limits_{0\leq t\leq T}\mathbb{E}|\xndfnc (t)|^q\Bigr)\cdot\delta_1^q,
\end{eqnarray}
and, by the Burkholder inequality,
\begin{eqnarray}
	&&\mathbb{E}|\tilde B_n^{df-RM}(t) - \tilde{\bar B}_n^{df-RM}(t)|^q \leq C  \mathbb{E} \Bigl| \int\limits_0^t\sum_{i=0}^{n-1} \left(b(U_i^{df}) - \tilde b(\bar U_i^{df})\right)\onei(s)\rd W(s) \Bigl|^q\notag \\
	&&+ C \mathbb{E} \Bigl| \int\limits_0^t\sum_{i=0}^{n-1} \Bigl(\int\limits_{t_i}^s \left( \mathcal{L}_{1,h}b(U_i^{df}) - \mathcal{L}_{1,h}\tilde b(\bar U_i^{df}) \right) \rd W(u)\Bigr)\onei(s)\rd W(s)\Bigl|^q\notag\\
	 &&\leq C\int\limits_0^t\sum\limits_{i=0}^{n-1}\mathbb{E}|b(U_i^{df})-\tilde b(\bar U_i^{df})|^q\onei(s)\rd s\notag\\
	 &&+C\int\limits_0^t\sum\limits_{i=0}^{n-1}(s-t_i)^{q/2}\cdot\mathbb{E}|\mathcal{L}_{1,h}b(U_i^{df})-\mathcal{L}_{1,h}\tilde b(\bar U_i^{df})|^q\onei(s)\rd s.
\end{eqnarray}
Note that
\begin{displaymath}
	\mathbb{E}|b(U_i^{df})-\tilde b(\bar U_i^{df})|^q\leq C\mathbb{E}|\tilde X^{df-RM}_n(t_i)-\xndfnc (t_i)|^q+C\delta_2^q\Bigl(1+\sup\limits_{0\leq t\leq T}\mathbb{E}|\xndfnc(t)|^q\Bigr).
\end{displaymath}
and, by Lemma \ref{DIV_DIFF_PROP2},
\begin{eqnarray}
	&&\mathbb{E}|\mathcal{L}_{1,h}b(U_i^{df})-\mathcal{L}_{1,h}{\tilde b}(\bar U_i^{df})|^q\leq C\Bigl(1+\sup\limits_{0\leq t\leq T}\mathbb{E}|\tilde X^{df-RM}_n(t)|^q+\sup\limits_{0\leq t\leq T}\mathbb{E}|\xndfnc(t)|^q\Big)\cdot h^q\notag\\
	&&\quad\quad\quad +K\mathbb{E}|\tilde X^{df-RM}_n(t_i)-\xndfnc(t_i)|^q\notag\\
	&&+C\Bigl(1+\sup\limits_{0\leq t\leq T}\mathbb{E}|\xndfnc(t)|^q\Bigr)\cdot (1+\delta_2^q)\cdot\left\{ \begin{array}{ll}
			\delta_2^q,  & \hbox{if} \  p_b\in\mathcal{K}^1_{\rm Lip}\\
			(\delta_2 h^{-1})^q, & \hbox{if} \ p_b\in\mathcal{K}^2.
			\end{array}\right.
\end{eqnarray}
Therefore, we get that for all $t\in [0,T]$
\begin{eqnarray}
\label{est_TBNBDFRM}
	&&\mathbb{E}|\tilde B^{df-RM}_n(t)-\tilde{ {\bar B}}^{df-RM}_n(t)|^q\leq C\int\limits_0^t\sum\limits_{i=0}^{n-1}\mathbb{E}|\tilde X^{df-RM}_n(t_i)-\tilde{ {\bar X}}^{df-RM}_n(t_i)|^q\onei(s)\rd s\notag\\
	&&\quad\quad\quad\quad+C\Bigl(1+\sup\limits_{0\leq t\leq T}\mathbb{E}|\xndfnc (t)|^q\Bigr)\cdot\delta_2^q\notag\\
	&&+C\Bigl(1+\sup\limits_{0\leq t\leq T}\mathbb{E}|\tilde X^{df-RM}_n(t)|^q+\sup\limits_{0\leq t\leq T}\mathbb{E}|\xndfnc(t)|^q\Big)\cdot h^{3q/2}\notag\\
	&&+C\Bigl(1+\sup\limits_{0\leq t\leq T}\mathbb{E}|\xndfnc(t)|^q\Bigr)\cdot (1+\delta_2^q)\cdot
			\left\{ \begin{array}{ll}
			(h^{1/2}\delta_2)^q,  & \hbox{if} \  p_b\in\mathcal{K}^1_{\rm Lip}\\
			(h^{-1/2}\delta_2)^q, & \hbox{if} \ p_b\in\mathcal{K}^2.
			\end{array}\right.
\end{eqnarray}
From \eqref{est_TANADFRM} and \eqref{est_TBNBDFRM} we get
\begin{eqnarray}
&&\mathbb{E}|\tilde X^{df-RM}_n(t)-\tilde{ {\bar X}}^{df-RM}_n(t)|^q\leq C\int\limits_0^t\sum\limits_{i=0}^{n-1}\mathbb{E}|\tilde X^{df-RM}_n(t_i)-\tilde{ {\bar X}}^{df-RM}_n(t_i)|^q\onei(s)\rd s\notag\\
	&&\quad\quad\quad\quad+C\Bigl(1+\sup\limits_{0\leq t\leq T}\mathbb{E}|\xndfnc (t)|^q\Bigr)\cdot(\delta_1^q+\delta_2^q)\notag\\
	&&+C\Bigl(1+\sup\limits_{0\leq t\leq T}\mathbb{E}|\tilde X^{df-RM}_n(t)|^q+\sup\limits_{0\leq t\leq T}\mathbb{E}|\xndfnc(t)|^q\Big)\cdot h^{3q/2}\notag\\
	&&+C\Bigl(1+\sup\limits_{0\leq t\leq T}\mathbb{E}|\xndfnc(t)|^q\Bigr)\cdot (1+\delta_2^q)\cdot
			\left\{ \begin{array}{ll}
			(h^{1/2}\delta_2)^q,  & \hbox{if} \  p_b\in\mathcal{K}^1_{\rm Lip}\\
			(h^{-1/2}\delta_2)^q, & \hbox{if} \ p_b\in\mathcal{K}^2.
			\end{array}\right.
\end{eqnarray}
Thereby, the Gronwall's lemma implies
\begin{eqnarray}
\label{tineq_est_rmdf5}
&&\sup\limits_{0\leq t \leq T}\|\tilde X^{df-RM}_n(t)-\tilde{ {\bar X}}^{df-RM}_n(t)\|_q\leq C_1\Bigl(1+\sup\limits_{0\leq t\leq T}\|\xndfnc (t)\|_q\Bigr)\cdot(\delta_1+\delta_2)\notag\\
	&&+C_2\Bigl(1+\sup\limits_{0\leq t\leq T}\|\tilde X^{df-RM}_n(t)\|_q+\sup\limits_{0\leq t\leq T}\|\xndfnc(t)\|_q\Big)\cdot h^{3/2}\notag\\
	&&+C_3\Bigl(1+\sup\limits_{0\leq t\leq T}\|\xndfnc(t)\|_q\Bigr)\cdot (1+\delta_2)\cdot
			\left\{ \begin{array}{ll}
			h^{1/2}\delta_2,  & \hbox{if} \  p_b\in\mathcal{K}^1_{\rm Lip}\\
			h^{-1/2}\delta_2, & \hbox{if} \ p_b\in\mathcal{K}^2.
			\end{array}\right.
\end{eqnarray}
Combining \eqref{tineq_est_rmdf4}, \eqref{tineq_est_rmdf5}, and Lemma \ref{BoundRanMilNo} we get the thesis. \ \ \ $\blacksquare$
%%%%%%%%%%%%%%%%%%%
\section{Lower bounds and optimality of the randomized derivative-free Milstein algorithm} \label{sec:low}
This section is dedicated to establishing lower bounds on the worst-case error of an arbitrary algorithm from $\Phi_n$ and to prove that the error of randomized derivative-free Milstein algorithm $\bar X^{RM}_n$ asymptotically attains optimality. 

\begin{lemma}
\label{lower_b_lem}
Let $q\in [2,+\infty)$, $\gamma_1, \gamma_2 \in (0,1]$, $K\in (0,+\infty)$, then 
\begin{align*}
	e_n^{(q)}(\mathcal{F}(\gamma_1,\gamma_2,q,K), W, V^i, \delta_1, \delta_2) & = \Omega(\max\{n^{-\min\{1/2 + \gamma_1, \gamma_2\})}, \delta_1, \delta_2\}),
\end{align*}
for $i=1,2$ as $n\to +\infty$, $\max\{\delta_1, \delta_2\} \to 0+$.
\end{lemma}
\begin{proof}
Firstly, we recall known results on lower bounds in the case of exact information, i.e. $\delta_1 = \delta_2 = 0$.
For Lebesgue integration of H\"older continuous functions and randomized standard information about integrand, accordingly to \cite{NOV}, we have 
\begin{displaymath}
	e_n^{(q)}(\mathcal{F}(\gamma_1,\gamma_2,q,K), W, V^i, 0, 0)= \Omega(n^{-(1/2 +  \gamma_1)}),
\end{displaymath}
for $i=1,2$.

The following lower bound is established in \cite{PrMo14} and \cite{Hein} for  It\^o integration
\begin{equation}
\label{lower_b_Ito}
e_n^{(q)}(\mathcal{F}(\gamma_1,\gamma_2,q,K), W, V^i, 0, 0) = \Omega (n^{-\gamma_2}),
\end{equation}
for $i=1,2$. (The lower bound \eqref{lower_b_Ito} holds also in the case when the evaluation points for $W$ are chosen in adaptive way, see \cite{Hein} for details.)

As the constant noise on the levels $\pm\delta_1$ and $\pm\delta_2$ is permissible for $a$ and $b$, respectively, we have 
\begin{displaymath}
	e_n^{(q)}(\mathcal{F}(\gamma_1,\gamma_2,q,K), W, V^i, \delta_1, \delta_2) = \Omega(\max\{\delta_1, \delta_2\}),
\end{displaymath}
for $i=1,2$, cf. proof of Lemma 3 in \cite{PMPP17}.

Therefore, as the worst-case error cannot be smaller than the error for subproblems, the proof is completed.
\end{proof}

The following theorem is the main result of the paper and establishes optimality of randomized derivative-free Milstein algorithm.
\begin{theorem}
\label{main_theorem}
Let $q\in [2,+\infty)$, $\gamma_1, \gamma_2 \in (0,1]$, $K\in (0,+\infty)$, then 
\begin{displaymath}
e_n^{(q)}(\mathcal{F}(\gamma_1,\gamma_2,q,K), W, V^1, \delta_1, \delta_2)  = \Theta(\max\{n^{-\min\{1/2 + \gamma_1, \gamma_2\})}, \delta_1, \delta_2\}),
\end{displaymath}
as $n\to +\infty$, $\max\{\delta_1,\delta_2\}\to 0+$. The optimal algorithm is randomized derivative-free Milstein algorithm $\bar X^{df-RM}_n$. 
\end{theorem}
Sharp bounds for the class $V^2$ in the case when $\delta_2>0$ remain as an open problem.
%%%%%%%%%%%%%%%%%%%
%%%%%%%%%%%%%%%%%%%
\section{Numerical experiments} \label{sec:num}
We present numerical results for the randomized derivative-free Milstein algorithm $\bar X^{RM}_n$ for the following problem
	\begin{equation*}
	\label{SIMUL_PROB}
			\left\{ \begin{array}{ll}
			dX(t)={\sin (M \cdot X(t) \cdot t ^ {\gamma_1}) \rd t}+{\cos (M \cdot X(t) \cdot t ^ {\gamma_2})\rd W(t)}, &t\in [0,T], \\
			X(0)=1.0,
		\end{array}\right.
	\end{equation*}
where $M = 100$, $\gamma_1$, $\gamma_2 = \min\{\gamma_1 + 0.5,1\}$. The drift and diffusion coefficients are H\"older continuous functions with the H\"older exponents $\gamma_1$ and $\gamma_2$, respectively. The expected theoretical convergence rate for this problem, accordingly to Theorem \ref{RMDF_NOIZZ_ERR}, is $n^{-\gamma_2}$ as $n$ tends to $+\infty$, and $\delta_1$, $\delta_2$ tend to zero.

Note that the exact solution of \eqref{SIMUL_PROB} is not known. Hence, in the simulations we computed in parallel the approximation of the solution for mesh of cardinality $n$ and $1000n$, treating the one on dense mesh as the exact. The rule of thumb for such a choice is as follows. The projected convergence rate is at least $n^{-0.5}$, so the error for $1000n$ should be at least an order of magnitude lower than the error on $n$ points, hence,
\begin{displaymath} 
	\| \bar X^{RM}_{1000n}(T) - \bar X^{RM}_n(T)\|_2 \approx \| X(T) - \bar X^{RM}_n(T)\|_2.
\end{displaymath}
The expectation is estimated as an average taken over $K = 10^4$ trajectories of the driving Wiener process.
The informational noise for the coefficients $a$ and $b$ is simulated as follows. We assume that the corrupting functions $p(t,y)$ for the drift and diffusion coefficients are bounded, i.e. $|p_a(t,y)| \leq \delta_1$ and $|p(t,y)|\leq \delta_2$. The noising procedure was simulated as a realization of a random variable uniformly distributed on $[0,1]$, scaled by the respective precision level $\delta_1$ or $\delta_2$. Each corruption was generated independently. The obtained results are presented in  Figure \ref{figure_1} and Figure \ref{figure_2}. For the obtained numerical results, the empirical convergence rate was also computed (as the linear regression of the $\log$ n vs $\log$ error curve), the summary of those can be find in the Table 1. 

\begin{figure}
\includegraphics[width=0.8\linewidth]{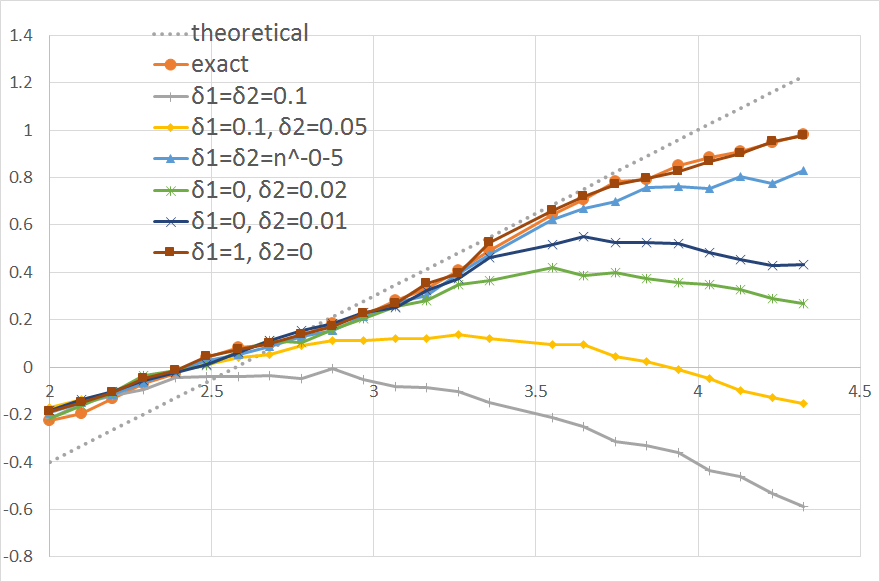}
\caption{Error for exact / noisy information for the case $\gamma_1 = 0.2$, $\gamma_2=0.7$.}\label{figure_1}
\end{figure}

\begin{figure}
\includegraphics[width=0.8\linewidth]{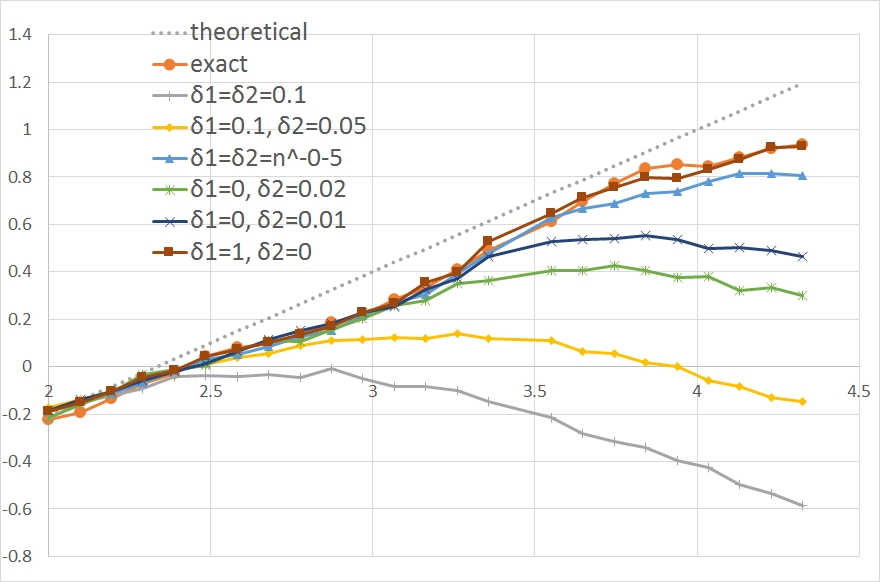}
\caption{Error for exact / noisy information for the case $\gamma_1 = 0.1$, $\gamma_2=0.6$.}\label{figure_2}
\end{figure}

\begin{table}[] \label{table}
\begin{tabular}{|l|c|c|}
\hline
                                                            & $\gamma_1=0.1$, $\gamma_2=0.6$ & $\gamma_1=0.2$, $\gamma_2=0.7$ \\ \hline
theoretical                                                 & 0.6            & 0.7            \\ \hline
exact                                                       & 0.54           & 0.55           \\ \hline
$\delta_1=\delta_2=0.1$                                                   & -0.20          & -0.19          \\ \hline
\begin{tabular}[c]{@{}l@{}}$\delta_1=0.1$,\\   $\delta_2=0.05$\end{tabular} & 0.01           & 0.00           \\ \hline
$\delta_1=\delta_2=n^{-0-5}$                               & 0.48           & 0.48           \\ \hline
$\delta_1=0, \delta_2=0.02$                                               & 0.25           & 0.23           \\ \hline
$\delta_1=0, \delta_2=0.01$                                               & 0.33           & 0.31           \\ \hline
$\delta_1=1, \delta_2=0$                                                  & 0.52           & 0.54           \\ \hline
\end{tabular}
\caption{Empirical convergence rates for various precision levels.}
\end{table}

The obtained numerical results confirm the theoretical results. The most surprising might be the fact that for a set precision on diffusion coefficient and with increasing number of discretization points, the error grows. That indicates that it is likely that the theoretical upper bound for error estimate for analyzed method is sharp with respect to the factor of $\delta_2 n^{1/2}$ in Theorem \ref{RMDF_NOIZZ_ERR}. This behavior is not observed for the set precision $\delta_1$ on the drift coefficient and increasing number of discretization points. 
The results also prove that with precision levels tending to zero with the theoretical convergence rate of the method, the observed convergence rate behaves similarly as by the exact information.

That clearly indicates also that this method cannot be optimal, as we can simply omit part of the information, not letting the error coming from the diffusion coefficient corruption increase the overall error of the method.

\section{Conclusions} \label{sec:con}
We investigated the $n$th minimal error for  pointwise approximation of scalar stochastic differential equations under inexact information about  drift and diffusion coefficients. We provided  implementable derivative-free randomized Milstein scheme $\bar X^{df_RM}_n$ and proved upper bounds on its error. It turned out that in some cases the algorithm $\bar X^{df_RM}_n$ is optimal.  We also reported  numerical experiments. They  confirmed obtained theoretical results.

In this paper we considered only noisy information about drift and diffusion coefficients. In the case when also the evaluations of the Wiener process are corrupted direct application of the technique used in this paper is not possible. Hence, further  extension of research on  the subject is needed  both for the lower and the upper bounds on the  error.
%%%%%%%%%%%%%%%%%%%%%%%%%%%%%%%%%%

%%%%%%%%%%%%%%%%%%%
%%%%%%%%%%%%%%%%%%%
\section{Appendix}
The proofs of the following two lemmas are straightforward and will be omitted.
\begin{lemma} 
\label{f_bounds}
If $f\in F^{\gamma}_K$, $\gamma\in\{\gamma_1,\gamma_2\}$, then for all $(t,y)\in [0,T]\times\mathbb{R}$
\begin{displaymath}
	|f(t,y)|\leq K_1 (1+|y|),
\end{displaymath}
\begin{displaymath}	
	\Bigl|\frac{\partial^j f}{\partial y^j}(t,y) \Bigl| \leq K, \quad j=1,2,
\end{displaymath} 
where $\displaystyle{K_1=K(1+\max\{T^{\gamma_1},T^{\gamma_2}\}})$.
\end{lemma}
%%%%%%%%%%%%%%%%
\begin{lemma}
\label{DIV_DIFF_PROP1} For all $n\in\mathbb{N}$, $b\in\mathcal{B}_K^{\gamma_2}$, and all $t\in [0,T]$, $y,z\in\mathbb{R}$ it holds
		\begin{displaymath}
			|L_1 b(t,y)| \leq K K_1 (1+|y|),
		\end{displaymath}
		\begin{displaymath}
			|\mathcal{L}_{1,h}b(t,y)|\leq KK_1(1+|y|),
		\end{displaymath}
		\begin{displaymath}
			|L_1b(t,y)- \mathcal{L}_{1,h}b(t,z)| \leq K|y-z|+KK_1(1+|z|)h,
		\end{displaymath}
		where $h=T/n$ and $\displaystyle{K_1=K(1+\max\{T^{\gamma_1},T^{\gamma_2}\}})$.
\end{lemma}
%{\bf Proof.\rm} We can notice that there exists $\xi_{t,y,h} \in (y,y+h)$ such that
%\begin{align*}
%	\left| \frac{\partial b}{\partial y}(t,y) - \frac{b(t,y+h) - b(t,y)}{h} \right| = & \left| \frac{\partial b}{\partial y}(t,y) - \frac{\partial b}{\partial y}(t,\xi_{t,y,h}) \right| \leq Kh,
%\end{align*}
%and hence
%\begin{align*}
%	|L_1b(t,y)- \mathcal{L}_{1,h}b(t,y)| = &  \left| b(t,y)\cdot\frac{\partial b}{\partial y}(t,y) - b(t,y)\cdot\frac{b(t,y+h) - b(t,y)}{h} \right| \\ 
%	= & |b(t,y)|\cdot \left|\frac{\partial b}{\partial y}-\frac{b(t,y+h)-b(t,y)}{h}\right| \leq KK_1(1+|y|)h.
%\end{align*}
%Therefore
%\begin{eqnarray}
%	&&|L_1b(t,y)- \mathcal{L}_{1,h}b(t,z)| \leq |L_1b(t,y)-L_1b(t,z)|+|L_1b(t,z)-\mathcal{L}_{1,h}b(t,z)| \notag\\
%	&&\quad\quad\quad\quad\leq K|y-z|+KK_1(1+|z|)h.
%\end{eqnarray}
%This ends the proof. \ \ \ $\blacksquare$
%%%%%%%%%%%%%%%%%%%
In the following lemma we investigate behavior of difference operator $\mathcal{L}_{1,h}$ in the case of inexact information about $b$.
\begin{lemma}
	\label{DIV_DIFF_PROP2} There exists a positive constant $C$, such that for all $n\in\mathbb{N}$, $\delta_1,\delta_2\in [0,1]$, $(a,b)\in\mathcal{A}^{\gamma_1}_K\times\mathcal{B}^{\gamma_2}_K$, $(\tilde a,\tilde b)\in V_a(\delta_1)\times (V^1_b(\delta_2)\cup V^2_b(\delta_2))$, and all $t\in [0,T]$, $y,z\in\mathbb{R}$ it holds
	\begin{equation}
	\label{an_bound1}
		|\tilde a(t,y)|\leq C(1+\delta_1)(1+|y|),
	\end{equation}
	\begin{equation}
	\label{bn_bound1}
		|\tilde b(t,y)|\leq C(1+\delta_2)(1+|y|),
	\end{equation}
		\begin{equation}
		\label{est_diff_op1}
|\mathcal{L}_{1,h}\tilde b(t,y)|  \leq C(1+\delta_2)(1+|y|)\cdot\left\{ \begin{array}{ll}
			1+\delta_2,  & \hbox{if} \  p_b\in\mathcal{K}^1_{\rm Lip},\\
			1+\delta_2h^{-1}, & \hbox{if} \ p_b\in\mathcal{K}^2,
			\end{array}\right.
		\end{equation}
\begin{eqnarray}
\label{est_diff_op2}
&&|\mathcal{L}_{1,h}b(t,y)-\mathcal{L}_{1,h}\tilde b(t,z)|  \leq C(1+|y|+|z|)h+K|y-z|\notag\\
&&\quad\quad\quad+C(1+|z|)\cdot (1+\delta_2)\cdot\left\{ \begin{array}{ll}
			\delta_2,  & \hbox{if} \  p_b\in\mathcal{K}^1_{\rm Lip}\\
			\delta_2 h^{-1}, & \hbox{if} \ p_b\in\mathcal{K}^2.
			\end{array}\right.
\end{eqnarray}
\end{lemma}
{\bf Proof. \rm} The proof of \eqref{an_bound1} and \eqref{bn_bound1} is straightforward.

We have that
\begin{equation}
\label{est_diff_L_1}
	|\mathcal{L}_{1,h}b(t,y)-\mathcal{L}_{1,h}\tilde b(t,z)|\leq |\mathcal{L}_{1,h}b(t,y)-\mathcal{L}_{1,h}b(t,z)|+|\mathcal{L}_{1,h}b(t,z)-\mathcal{L}_{1,h}\tilde b(t,z)|.
\end{equation}
From Lemma \ref{DIV_DIFF_PROP1} we get that
\begin{eqnarray}
\label{est_diff_L_2}
	&&|\mathcal{L}_{1,h}b(t,y)-\mathcal{L}_{1,h}b(t,z)|\leq |\mathcal{L}_{1,h}b(t,y)-L_1b(t,z)| + |L_1b(t,z)-\mathcal{L}_{1,h}b(t,z)|\notag\\
	&&\quad\quad\quad\quad \leq K|y-z|+K_1K(1+|y|)h+K_1K(1+|z|)h\notag\\
	&&\quad\quad\quad\quad\leq C(1+|y|+|z|)h+K|y-z|.
\end{eqnarray}
Furthermore,
\begin{eqnarray}
	&&|\mathcal{L}_{1,h}b(t,z)-\mathcal{L}_{1,h}\tilde b(t,z)|\leq |b(t,z)|\cdot |\Delta_h b(t,z)-\Delta_h \tilde b(t,z)|+\delta_2\cdot |p_b(t,z)|\cdot |\Delta_h \tilde b(t,z)|\notag\\
	&&\quad\quad\quad\quad \leq K_1(1+|z|)\cdot |\Delta_h b(t,z)-\Delta_h \tilde b(t,z)|+\delta_2\cdot (1+|z|)\cdot |\Delta_h \tilde b(t,z)|.
\end{eqnarray}
Note that
\begin{equation}
\label{est_diff_L_4}
	|\Delta_h \tilde b(t,z)|\leq |\Delta_h b(t,z)|+\delta_2\cdot |\Delta_h p_b(t,z)|\leq K+\delta_2\cdot|\Delta_h p_b(t,z)|.
\end{equation}
Moreover,
\begin{displaymath}
	|\Delta_h b(t,z)-\Delta_h \tilde b(t,z)|=\delta_2\cdot |\Delta_h p_b(t,z)|,
\end{displaymath}
and
\begin{equation}
\label{est_diff_L_5}
	|\Delta_h p_b(t,z)|\leq \left\{ \begin{array}{ll}
			1,  & \hbox{if} \  p_b\in\mathcal{K}^1_{\rm Lip}\\
			2h^{-1}, & \hbox{if} \ p_b\in\mathcal{K}^2.
			\end{array}\right.
\end{equation}
Hence,
\begin{eqnarray}
\label{est_diff_L_3}
	&&|\mathcal{L}_{1,h}b(t,z)-\mathcal{L}_{1,h}\tilde b(t,z)|\leq (K_1+\delta_2)\cdot\delta_2\cdot(1+|z|)\cdot|\Delta_h p_b(t,z)|+K\cdot\delta_2\cdot(1+|z|)\notag\\
	&&\quad\quad\quad\quad\leq C (1+|z|)\cdot (1+\delta_2)\cdot\left\{ \begin{array}{ll}
			\delta_2,  & \hbox{if} \  p_b\in\mathcal{K}^1_{\rm Lip}\\
			\delta_2\cdot h^{-1}, & \hbox{if} \ p_b\in\mathcal{K}^2.
			\end{array}\right.
\end{eqnarray}
Combining \eqref{est_diff_L_1}, \eqref{est_diff_L_2}, and \eqref{est_diff_L_3} we get \eqref{est_diff_op2}. Finally, by \eqref{bn_bound1}, \eqref{est_diff_L_4}, \eqref{est_diff_L_5}, and 
\begin{eqnarray}
	|\mathcal{L}_{1,h}\tilde b(t,y)|\leq C(1+\delta_2)\cdot (1+|y|)\cdot (K+|\Delta_h p_b(t,y)|)
\end{eqnarray}
the result \eqref{est_diff_op1} follows. \ \ \ $\blacksquare$
%%%%%%%%%%%
%%%%%%%%%%%%%
\begin{lemma}
\label{BoundRanMilNo}
\begin{itemize}
\item [(i)] There exists a positive constant $C$, depending only on the parameters of the class $\mathcal{F}(\gamma_1,\gamma_2,q,K)$, such that for all $n\in\mathbb{N}$, $(a,b,\eta)\in \mathcal{F}(\gamma_1,\gamma_2,q,K)$, we have
\begin{equation}
	\label{est_ran_mil_no1}
	\sup_{t\in[0,T]} \mathbb{E} |\tilde{X}_n^{RM}(t)|^q \leq C,
\end{equation}
\begin{equation}
	\label{est_ran_mil_no2}
	\sup_{t\in[0,T]} \mathbb{E} |\tilde{X}_n^{df-RM}(t)|^q \leq C.
\end{equation}
\item [(ii)] There exists a positive constant $C$, depending only on the parameters of the class $\mathcal{F}(\gamma_1,\gamma_2,q,K)$, such that for all $n\in\mathbb{N}$, $\delta_1, \delta_2\in [0,1]$, $(a,b,\eta)\in \mathcal{F}(\gamma_1,\gamma_2,q,K)$, $(\tilde a, \tilde b)\in V_a(\delta_1)\times V^1_b(\delta_2)$ , we have
\begin{equation}
	\label{est_ran_mil_no31}
	\sup_{t\in[0,T]} \mathbb{E} |\tilde{\bar X}_n^{df-RM}(t)|^q \leq C(1+\delta_1^q + \delta_2^q+ \delta_2^{2q}) \ e^{C T(1+\delta_1^q+\delta_2^q+\delta_2^{2q})}.
\end{equation}
\item [(iii)]
There exists a positive constant $C$, depending only on the parameters of the class $\mathcal{F}(\gamma_1,\gamma_2,q,K)$ and $q$, such that for all $n\in\mathbb{N}$, $\delta_1, \delta_2\in [0,1]$, $(a,b,\eta)\in \mathcal{F}(\gamma_1,\gamma_2,q,K)$, $(\tilde a, \tilde b)\in V_a(\delta_1)\times V^2_b(\delta_2)$ , we have
\begin{equation}
	\label{est_ran_mil_no32}
	\sup_{t\in[0,T]} \mathbb{E} |\tilde{\bar X}_n^{df-RM}(t)|^q\leq C(1+\delta_1^q + \delta_2^q+(1+\delta_2^q)\delta_2^qn^{q/2}) \ e^{C T(1+\delta_1^q+\delta_2^q+(1+\delta_2^q)\delta_2^qn^{q/2})}.
\end{equation}
\end{itemize}
\end{lemma}
{\bf Proof.\rm} We only show (ii) and (iii), since the proof of (i) is analogous. 

Take $(a,b,\eta)\in\mathcal{F}(\gamma_1,\gamma_2,q,K)$, $(\tilde a, \tilde b)\in V_a(\delta_1)\times (V^1_b(\delta_2)\cup V^2_b(\delta_2))$. By Lemma \ref{PROG_RMDFNO_PROC} we have that the random variables $\tilde b(t_i, \tilde{\bar X}_n^{df-RM}(t_i))$, $\mathcal{L}_{1,h}\tilde b(t_i, \tilde{\bar X}_n^{df-RM}(t_i))$ are $\tilde\Sigma^n_{t_i}$-measurable, while the increment $W(t)-W(t_i)$ is independent of $\tilde\Sigma^n_{t_i}$ for all $i=0,1,\ldots,n-1$ and $t\in [t_i,t_{i+1}]$. Additionally, $\|W(t)-W(t_i)\|_q= m_q\cdot (t-t_i)^{1/2}$, and $\|\mathcal{I}_{t_i,t}(W,W)\|_q\leq\frac{1}{2}(m_{2q}^2+1)(t-t_i)$ for $t\in [t_i,t_{i+1}]$, where $m_q$ is the $q$-th root of the $q$-th absolute moment of a normal variable with zero mean and variance equal to $1$. This and Lemma \ref{DIV_DIFF_PROP2} give, for all $i=0,1,\ldots,n-1$ and  $t\in[t_i, t_{i+1}]$, that
\begin{eqnarray}
\label{aux_est_1}
\|\tilde{\bar X}_n^{df-RM}(t) - \tilde{\bar X}_n^{df-RM}(t_i)\|_q &\leq& \|\tilde a(\xi_i, \tilde{\bar X}_n^{df-RM}(t_i))\|_q\cdot(t-t_i)\notag\\
&& + \|\tilde b(t_i, \tilde{\bar X}_n^{df-RM}(t_i))\|_q \cdot \|W(t)-W(t_i)\|_q\notag \\
&& + \|\mathcal{L}_{1,h}\tilde b(t_i, \tilde{\bar X}_n^{df-RM}(t_i))\|_q\cdot \|\mathcal{I}_{t_i,t}(W,W)\|_q\notag\\
&\leq & C\cdot(1+\delta_1+\delta_2)\cdot(1+\delta_2\cdot\max\{1,h^{-1}\})\times\notag\\
&& (1+\|\tilde{\bar X}_n^{df-RM}(t_i)\|_q)\cdot(t-t_i)^{1/2},
\end{eqnarray} 
where $C>0$ depends only on the parameters of the class $\mathcal{F}(\gamma_1,\gamma_2,q,K)$. Since $\|\tilde{\bar X}_n^{df-RM}(0)\|_q = \|\eta\|_q\leq\|\eta\|_{2q}\leq K$, we get by (\ref{aux_est_1}) and induction that
\begin{equation}
\label{aux_est_2}
	\max\limits_{i\in\{0,1,\ldots,n\}} \|\tilde{\bar X}_n^{df-RM} (t_i)\|_q<+\infty.
\end{equation}
From (\ref{aux_est_1}) and (\ref{aux_est_2}) we get that $\displaystyle{\sup\limits_{t\in[0,T]} \|\tilde{\bar X}_n^{df-RM}(t)\|_q < + \infty}$. Therefore, the function $[0,T]\ni t \mapsto \sup\limits_{0\leq u\leq t} \mathbb{E} |\tilde{\bar X}_n^{df-RM}(u)|^q\in\mathbb{R}_+\cup\{0\}$ is Borel measurable (as a nondecreasing function) and bounded. We now show that we can bound this mapping from above by a finite number that depends only on the parameters of the class $\mathcal{F}(\gamma_1,\gamma_2,q,K)$, $\delta_1$, and $\delta_2$.

We have that for all $t\in [0,T]$
\begin{displaymath}
	\mathbb{E}|\tilde{\bar X}_n^{df-RM}(t)|^q\leq C(\mathbb{E}|\eta|^q+\mathbb{E}|\tilde{\bar A}_n^{df-RM}(t)|^q+\mathbb{E}|\tilde{\bar B}_n^{df-RM}(t)|^q).
\end{displaymath}
From the H\"older inequality we obtain that
\begin{displaymath}
	\mathbb{E}|\tilde{\bar A}_n^{df-RM}(t)|^q\leq C_1(1+\delta_1^q)+C_2(1+\delta_1^q)\int\limits_0^t\sum\limits_{i=0}^{n-1}\mathbb{E}|\tilde{\bar X}_n^{df-RM}(t_i)|^q\onei(s)\rd s.
\end{displaymath}
Moreover, by the Burkholder inequality
\begin{eqnarray}
	&&\mathbb{E}|\tilde{\bar B}_n^{df-RM}(t)|^q
%\leq C\mathbb{E}\Bigl|\int\limits_0^t\sum\limits_{i=0}^{n-1}\tilde b(\bar U_i^{df})\cdot\mathbf{1}_{[t_i,t_{i+1})}(s)\rd W(s)\Bigl|^q\notag\\
%	&&\quad\quad +C\mathbb{E}\Bigl|\int\limits_0^t\sum\limits_{i=0}^{n-1}\Bigl(\int\limits_{t_i}^s\mathcal{L}_{1,h}\tilde b(\bar U_i^{df})\rd W(u)%\Bigr)\cdot\mathbf{1}_{[t_i,t_{i+1})}(s)\rd W(s)\Bigl|^q\notag\\
	\leq C_3(1+\delta_2^q)+C_4(1+\delta_2^q)\int\limits_0^t\sum\limits_{i=0}^{n-1}\mathbb{E}|\tilde{\bar X}_n^{df-RM}(t_i)|^q\onei(s)\rd s\notag\\
	&&+C_5 h^{q/2} \int\limits_0^t\sum\limits_{i=0}^{n-1}\mathbb{E}|\mathcal{L}_{1,h}\tilde b(\bar U_i^{df})|^q\onei(s)\rd s,
\end{eqnarray}
where, by Lemma \ref{DIV_DIFF_PROP2}, we have
\begin{equation}
\label{est_L_1}
	|\mathcal{L}_{1,h}\tilde b(\bar U_i^{df})|^q\leq C_1\cdot (1+\delta_2^q)\cdot(1+|\tilde{\bar X}_n^{df-RM}(t_i)|^q)\cdot\left\{ \begin{array}{ll}
			1+\delta_2^q,  & \hbox{if} \  p_b\in\mathcal{K}^1_{\rm Lip}\\
			1+\delta_2^q\cdot h^{-q}, & \hbox{if} \ p_b\in\mathcal{K}^2.
			\end{array}\right.
\end{equation}
Therefore, if $p_b\in\mathcal{K}^1_{\rm Lip}$ we get
\begin{eqnarray}
	&&\mathbb{E}|\tilde{\bar B}_n^{df-RM}(t)|^q\leq C_1(1+\delta_2^q)+C_2(1+\delta_2^{2q})\notag\\
	&&\quad\quad\quad\quad +C_3(1+\delta_2^q+\delta_2^{2q})\cdot\int\limits_0^t\sum\limits_{i=0}^{n-1}\mathbb{E}|\tilde{\bar X}_n^{df-RM}(t_i)|^q \onei(s)\rd s
\end{eqnarray}
while for $p_b\in\mathcal{K}^2$ it holds that
\begin{eqnarray}
	&&\mathbb{E}|\tilde{\bar B}_n^{df-RM}(t)|^q\leq C_1(1+\delta_2^q)\cdot (1+\delta_2^{q}\cdot h^{-q/2})\notag\\
	&&\quad\quad +C_3(1+\delta_2^q)\cdot (1+\delta_2^q\cdot h^{-q/2})\cdot\int\limits_0^t\sum\limits_{i=0}^{n-1}\mathbb{E}|\tilde{\bar X}_n^{df-RM}(t_i)|^q \onei(s)\rd s.
\end{eqnarray}
By applying the Gronwall's lemma we get the thesis in  (ii) and (iii). \ \ \ $\blacksquare$
\\ \\
%%%%%%%%%%%%%%%%%%%%%%%%%%%%%%%%%%
Finally, we recall the well-known bound on the absolute $L^{2q}$-moment of the solution $X$ of \eqref{PROBLEM}. The following lemma is a direct consequence of Theorems 4.3 and 4.4 in Chapter 2 in \cite{Mao11}.
\begin{lemma}	
\label{moment_bound_sol}
	There exists a positive constant $C$, depending only on the parameters of the class $\mathcal{F}(\gamma_1,\gamma_2,q,K)$, such that for all  $(a,b,\eta)\in \mathcal{F}(\gamma_1,\gamma_2,q,K)$, we have
\begin{displaymath}
	 \Bigl\|\sup\limits_{t\in [0,T]} |X(t)|\Bigr\|_{2q}\leq C.
\end{displaymath}
\end{lemma}
%%%%%%%%%%%%%%%%%%%%%%%%%%%%%%%%%%
%%%%%%%%%%%%%%%%%%%%%%%%%%%%%%%%%%
{\bf Acknowledgments.}
This research was supported by the National Science Centre, Poland, under project 2017/25/B/ST1/00945.

\end{document}